\documentclass[a4paper, 11pt]{article}
\usepackage{amsmath,amssymb,esint,amscd,xspace,bm,fancyhdr,color,authblk,srcltx,fontenc,bbm}
\usepackage{hyperref}
\usepackage{cite}
\usepackage{graphicx,float}
  
\newtheorem{theorem}{Theorem}[section]

\newtheorem{definition}[theorem]{Definition}

\newtheorem{lemma}[theorem]{Lemma}

\newtheorem{remark}[theorem]{Remark}
\newenvironment{proof}[1][Proof]{\textbf{#1.} }{\hfill\rule{0.5em}{0.5em}}
{\catcode`\@=11\global\let\AddToReset=\@addtoreset
\AddToReset{equation}{section}

\AddToReset{theorem}{section}

\newtheorem{taggedassumptionx}{Assumption}
\newenvironment{taggedassumption}[1]
 {\renewcommand\thetaggedassumptionx{#1}\taggedassumptionx}
 {\endtaggedassumptionx}

\title{Regularity for the steady Stokes-type flow of incompressible Newtonian fluids in some generalized function settings}

\author{Minh-Phuong Tran\footnote{Corresponding author} \thanks{Applied Analysis Research Group, Faculty of Mathematics and Statistics, Ton Duc Thang University, Ho Chi Minh city, Vietnam; \texttt{tranminhphuong@tdtu.edu.vn}}, Thanh-Nhan Nguyen\thanks{Group of Analysis and Applied Mathematics, Department of Mathematics, Ho Chi Minh City University of Education, Ho Chi Minh city, Vietnam; \texttt{nhannt@hcmue.edu.vn}}, Hong-Nhung Nguyen\thanks{Department of Mathematics, Ho Chi Minh City University of Education, Ho Chi Minh city, Vietnam; \texttt{nhungnh2102@gmail.com}}}
\date{\today} 

\begin{document}
\maketitle

\begin{abstract}
A study of regularity estimate for weak solution to generalized stationary Stokes-type systems involving $p$-Laplacian is offered. The governing systems of equations are based on steady incompressible flow of a Newtonian fluids. This paper also provides a relatively complete picture of our main results in two regards: problems with nonlinearity is regular with respect to the gradient variable; and asymtotically regular problems, whose nonlinearity satisfies a particular structure near infinity. For such Stokes-type systems, we derive regularity estimates for both velocity gradient and its associated pressure in two special classes of function spaces: the generalized Lorentz and $\psi$-generalized Morrey spaces.
\medskip

\medskip

\noindent{\emph{Mathematics Subject Classification.}} 35D30, 35Q35, 76D03, 42B25, 46E30.

\medskip

\noindent {\emph{Keywords:}}  Regularity; Stokes system; Fractional maximal operators; Generalized Lorentz spaces, $\psi$-Morrey spaces; Distribution functions, Asymptotically regular problem.

\end{abstract} 

                  
\section{Introduction and results}\label{sec:intro}


In this paper, we are concerned with the study of regularity for weak solutions to stationary Stokes system involving the $p$-Laplace operator under non-homogeneous Dirichlet boundary condition:
\begin{align}\label{eq:Stokes-p}
\begin{cases}
-\mathrm{div} \left((\sigma^2+|\nabla \mathbf{u}|^2)^{\frac{p-2}{2}}\nabla \mathbf{u}\right) + \nabla \mathrm{\pi} &= -\mathrm{div}\left(|\mathbf{f}|^{p-2}\mathbf{f}\right) \quad \ \text{in}\ \Omega, \\
 \hspace{2.5cm} \mathrm{div} (\mathbf{u}) &= \ 0  \hspace{2.48cm} \ \text{in}\ \Omega, \\
 \hspace{2.9cm} \mathbf{u} & = \ \mathbf{g} \hspace{2.58cm} \text{on}\ \partial\Omega, 
 \end{cases}
\end{align}
which is defined in an open bounded sub-domain $\Omega$ in $\mathbb{R}^d$, for $d \ge 2$ with non-smooth boundary; the source terms $\mathbf{f} \in L^p(\Omega;\mathbb{R}^{d^2})$, and $\mathbf{g} \in W^{1,p}(\Omega;\mathbb{R}^{d})$ with $p>1$. It is interesting to note that a new degeneracy parameter $\sigma \in [0,1]$ introduced to establish results in both degenerate/non-degenerate regimes: $\sigma=0$ for degenerate case and $\sigma \neq 0$ for the non-degenerate one. In this model, we consider a pair of unknown functions: $\mathbf{u}$ is the velocity vector $\mathbf{u}=(u^1, u^2, ..., u^d)^T:  \Omega \to \mathbb{R}^d$ and its associated pressure $\mathrm{\pi}: \Omega \to \mathbb{R}$.  More specifically, global regularity estimates for the weak solution pair of this problem are established in some generalized weighted function spaces via fractional maximal operators. 

The primary motivation for the study of this generalized Stokes system is in the realm of fluid mechanics with pressure- and shear-dependent viscosities. Problem~\eqref{eq:Stokes-p} is a generalization of the classical (linear) Stokes system governed by Laplace operator. A bit more precise, it is driven by the $p$-Laplacian with an additional degeneracy parameter $\sigma>0$. Also, it is a simplified stationary form of the generalization of nonsteady Navier-Stokes system without the convective term:
\begin{align*}
\frac{\partial \mathbf{u}}{\partial t} - \mathrm{div}\left((\sigma^2+|\nabla \mathbf{u}|^2)^{\frac{p-2}{2}}\nabla \mathbf{u}\right) + \nabla \mathrm{\pi} &= -\mathrm{div}\left(|\mathbf{f}|^{p-2}\mathbf{f}\right) \quad \ \text{in}\ \Omega, \\
\hspace{2.5cm} \mathrm{div} (\mathbf{u}) &= \ 0  \hspace{2.48cm} \ \text{in}\ \Omega,
\end{align*}
that is completed by the initial and boundary conditions. To be very closely linked to the $p$-Laplace equations, this system therefore has attracted a lot of attention of both physical and mathematical communities. For further readings as well as connections of Stokes systems with physical models, one could mention~\cite{FMR2005,MMS2005,FS2000,Solonikov2001} and references therein. A point worth emphasizing is that in the context of incompressible non-Newtonian fluids, as the power-law introduced by Lady\v{z}enskaya in~\cite{Ladyzenskaya}, deviatoric stress tensor may depend in a nonlinear way by $\mathbf{Du} = \frac{1}{2}\left(\nabla \mathbf{u} + \nabla^T \mathbf{u}\right)$-the symmetric part of the velocity gradient $\nabla\mathbf{u}$, instead of the full gradient $\nabla \mathbf{u}$ that we are interested in. From viewpoint of physical reality and fluid mechanics applications, the Stokes system depending on $\mathbf{Du}$ is more relevant. It appeared in famous models of incompressible non-Newtonian fluid flows such as blood, custard, ketchup, paint, toothpaste, etc (see~\cite{MRR1995, MR2005,CR2008} for physical models). However, the main viewpoint linking the results of this paper is that we extend the previous results~\cite{NP20,NP21a,NPN2022} into the disciplines of the stationary Navier-Stokes systems in the context of Newtonian fluids. With this strategy in mind, it is natural towards the forthcoming results, which aim to extend our main concern here to the flow of non-Newtonian fluids.

There have been notable research activities regarding solvability and the regularity properties of solutions to Stokes system~\eqref{eq:Stokes-p}. Notably, a special interest in the linear case $p=2$ has a wide-ranging applications in mathematical modeling and simulation of fluid flows. This specific case becomes an important subject studied by many authors. The early analysis of such system traced back to Lady\v{z}enskaya~\cite{Ladyzenskaya}, Lions~\cite{Lions1996} and Sobolevski\v{i}~\cite{Sobolevskii}. Later, regularity theory of this system has received a lot of attention by multiple authors. Regarding the classical formulation of Stokes, the set of results be presented in previous papers~\cite{Breit2013,Dong2018,DK2018_2,MR2006,MR2007} have been driven for both and global estimates in various functional settings. Recently, this line of research has subsequently been ongoing to develop for generalized stationary Stokes systems. To mention a few, the reader is referred to~\cite{DJS2011,DKS2014,GM82} that have dealt with the interior Morrey and Campanato regularity estimates for Stokes systems driven by $p$-Laplacian. Further, global weighted and unweighted $L^q$-estimates have studied in~\cite{DK2013,BS2017}, or~\cite{GS2015} for $W^{1,p}$-estimates,~\cite{Macha2011} for H\"older regularity, and we also address our recent work in~\cite{NPN2022} for the generalized Lorentz and $\psi$-generalized Morrey estimates in the more general linear case. Other progress on the stationary Stokes system has been achieved in the last few years, we herein refer the interested reader to~\cite{Shen1995,GSS1994,FKV1988,CL2017} for further discussion in related contexts. 

There are two main topics in this paper: on the one hand, this is a continuation of previous work~\cite{NPN2022}, we obtain global Calder\'on-Zygmund-type estimates for gradient of weak solution and associated pressure to~\eqref{eq:Stokes-p} in both the generalized Lorentz spaces with Muckenhoupt weights and $\psi$-generalized Morrey spaces. More than that, our results shown here also hold for a more general class of nonlinearities and operators of elliptic type in Stokes systems (with degeneracy $\sigma>0$) than the one in~\eqref{eq:Stokes-p}. Therefore, it is natural to investigate a more general stationary Stokes-type system with quasilinear elliptic operators of the type
\begin{align}\label{eq:Stokes}
\begin{cases}
-\mathrm{div} \left(\mathbf{A}(x,\nabla \mathbf{u})\right) + \nabla \mathrm{\pi} &= -\mathrm{div}\left(\mathbf{B}(x,\mathbf{f})\right) \quad \, \text{in}\ \Omega, \\
 \hspace{1.5cm} \mathrm{div} (\mathbf{u}) &= \ 0  \hspace{2.46cm} \ \text{in}\ \Omega, \\
 \hspace{1.9cm} \mathbf{u} & = \ \mathbf{g} \hspace{2.56cm} \text{on}\ \partial\Omega,
 \end{cases}
\end{align}
where $\mathbf{A}:  \Omega \times \mathbb{R}^{d^2} \to \mathbb{R}^{d^2}$ is a $C^1$-Carath\'eodory matrix-valued function satisfying the following ellipticity and boundedness conditions: there exist $p>1$, $\Upsilon>0$ such that
\begin{align}\label{cond:A1}
 \left| \mathbf{A}(x,\nu) \right| + |\partial_{\nu} \mathbf{A}(x,\nu)||\nu| & \le \Upsilon \left(\sigma^2 + |\nu|^2 \right)^{\frac{p-1}{2}}, \\
\label{cond:A2}
\left\langle \mathbf{A}(x,\nu_1)-\mathbf{A}(x,\nu_2),  \nu_1 - \nu_2 \right\rangle & \ge \Upsilon^{-1} \left(\sigma^2 + |\nu_1|^2 + |\nu_2|^2 \right)^{\frac{p-2}{2}}|\nu_1 - \nu_2|^2,
\end{align}
for any $\nu$, $\nu_1$, $\nu_2$ in $\mathbb{R}^{d^2}$  with a given constraint $\sigma^2 + |\nu_1|^2 + |\nu_2|^2 \neq 0$ and almost every $x$ in $\Omega$, and a fixed parameter $\sigma \in [0,1]$ allows us to discriminate between degenerate and non-degenerate cases. Regarding the non-degenerate case, it can be reduced to $\sigma=1$ by changing the value of $\Upsilon$ in~\eqref{cond:A1} and~\eqref{cond:A2}. The nonlinearity $\mathbf{A}$ is modeled upon the following matrix-valued function 
$$\mathbf{A}(x,\nu) = (\sigma^2+|\nu|^2)^{\frac{p-2}{2}}\nu, \quad (x,\nu) \in \Omega \times \mathbb{R}^{d^2},$$ that gives rise to the degenerate/non-degenerate $p$-Laplacian system~\eqref{eq:Stokes-p}. For the right-hand side divergence form, in dimensional analogy with~\eqref{cond:A1}, we have a Carath\'eodory matrix-valued function $\mathbf{B}: \Omega \times \mathbb{R}^{d^2} \to \mathbb{R}^{d^2}$ satisfying the growth assumption
\begin{align}\label{cond:Bf}
 \left| \mathbf{B}(x,\nu) \right| & \le \Upsilon \left(\sigma^2 + |\nu|^2 \right)^{\frac{p-1}{2}},
\end{align}
for every $x \in \Omega$ and $\nu \in \mathbb{R}^{d^2}$. It is to be noticed that in the general context of quasilinear elliptic equations (involving the $p$-Laplacian), there have been a large number of studies falling into the scope of regularity theory by making the use of different approaches. Here, we send the reader to~\cite{AM2007,BW2004,CP1998,Giusti2003,DJ1993,Min2010,Min2011,Iwaniec,Tran2019,PN20a} and references therein for results as well as technical tools associated to it. 

Before formulating the main result regarding the first topic, let us remind the reader the definition of the weak solution to the system. For the uniqueness purposes, we consider a pair of weak solutions $(\mathbf{u},\mathrm{\pi})$ to~\eqref{eq:Stokes} in two sets as below, respectively:
\begin{align}\notag 
W^{1,q}_{0,\mathrm{div}}(\Omega;\mathbb{R}^d) := \left\{\mathbf{v} \in W^{1,q}_0(\Omega;\mathbb{R}^d): \ \mathrm{div}(\mathbf{v}) = 0  \mbox{ in }  \Omega\right\},
\end{align}
for $q >1$, and 
\begin{align}\notag 
L^q_{\mathrm{int}}(\Omega) := \left\{\varphi \in L^q(\Omega): \ \int_{\Omega} \varphi(z)dz = 0\right\}.
\end{align}
Further, for every $\mathbf{h} \in W^{1,q}(\Omega;\mathbb{R}^d)$, we also define here 
$$W^{1,q}_{\mathbf{h},\mathrm{div}}(\Omega;\mathbb{R}^d) := \mathbf{h} + W^{1,q}_{0,\mathrm{div}}(\Omega;\mathbb{R}^d)$$ 
for the sake of readability. The notion of weak solution pair to our system~\eqref{eq:Stokes} is well tailored for our purpose: first handle with the gradient weak solutions leaving the pressure as an obstacle and then return to find the associated pressure from the frame of functional analysis stemming from the system. The idea is patterned on the approach of~\cite{BS2017, DK2018_2, Dong2018}. 
\begin{definition}
Let $\mathbf{f} \in L^p(\Omega;\mathbb{R}^{d^2})$ and $\mathbf{g} \in W^{1,p}(\Omega;\mathbb{R}^{d})$. We shall say that a pair 
$$(\mathbf{u},\mathrm{\pi}) \in W^{1,p}_{\mathbf{g},\mathrm{div}}(\Omega;\mathbb{R}^d) \times L^{p'}_{\mathrm{int}}(\Omega), \quad p' := \frac{p}{p-1},$$ is a weak solution to~\eqref{eq:Stokes} if it satisfies the following properties:
\begin{itemize}
\item[(i)] For any test function $\varphi \in W^{1,p}_{0,\mathrm{div}}(\Omega;\mathbb{R}^d)$, it holds
\begin{align}\label{var-form1}
\int_{\Omega} \langle \mathbf{A}(x, \nabla \mathbf{u}), \nabla \varphi \rangle dx = \int_{\Omega} \langle \mathbf{B}(x,\mathbf{f}), \nabla \varphi \rangle dx.
\end{align}
\item[(ii)] Associated with a concrete weak solution $\mathbf{u}$ satisfying~\eqref{var-form1}, it also holds
\begin{align}\label{var-form2}
\int_{\Omega} \langle \mathbf{A}(x, \nabla \mathbf{u}), \nabla \varphi \rangle dx - \int_{\Omega} \mathrm{\pi}\mathrm{div}(\varphi) dx = \int_{\Omega} \langle \mathbf{B}(x,\mathbf{f}), \nabla \varphi \rangle dx,
\end{align}
for any test function $\varphi \in W^{1,p}_0(\Omega;\mathbb{R}^d)$. Then, $\mathrm{\pi}$ is so-called the associated pressure of $\mathbf{u}$.
\end{itemize}
\end{definition}

To the question of clarifying the existence of weak solution pair $(\mathbf{u},\mathrm{\pi})$ to this problem on a bounded domain under some appropriated structure conditions, we send the reader to~\cite{Ladyzenskaya1,Lions1969, Ladyzenskaya2,DRW2010} or a large number of connections with similar topics for further research or reading. As aforementioned, our results fall into two categories: in the setting of generalized Lorentz spaces and $\psi$-generalized Morrey spaces.  It is known that the classical Lorentz spaces $L^{\mathfrak{s},\mathfrak{t}}$ originated in~\cite{Lorentz1950}, can be seen as generalization of classical Lebesgue spaces and have become fashionable in mathematical analysis in years. The generalized Lorentz spaces we mention in this paper, denoted by $\mathcal{L}^{\mathfrak{s},\mathfrak{t}}_{\mu,\omega}(\Omega)$, also known as \emph{weighted Lorentz spaces with two parameters}, as one of the normability result for weighted Lorentz spaces. Theory of such generalized Lorentz spaces can be traced back as far as the works of Carro and Soria in~\cite{CS1993} and further developed in~\cite{CS1997, CRS07} (see Section~\ref{sec:notationanddef} below for the definition). Next, as a consequence, it comes out a global bound for solutions in generalized Morrey spaces. To our knowledge, there have been various known types of generalizations of classical Morrey spaces, see for instance~\cite{AM1997, MS1979,RSS2013,DN1982} and the references therein. And in this regard, we devote a special attention to the $\psi$-generalized Morrey spaces,  often denoted as $\mathrm{M}^{s,\psi}(\Omega)$, and will also be lightened below. 

A point worth emphasizing in our study is the use of \emph{weighted fractional maximal distribution functions} (WFMDs) for resulting estimates (see Definition~\ref{def:wfmds} from below). The method of using WFMDs is a key ingredient that mainly based on level-set inequalities involving distribution functions and properties of fractional maximal operators $\mathcal{M}_\alpha$. This technique arose in our study when providing regularity in terms of $\mathcal{M}_\alpha$, cf.~\cite{PN20a,PN20b,PN19b}. The approach relies on estimates over level-sets is motivated by a standard technique proposed by  by Acerbi and Mingione in~\cite{AM2007} and related ideas in~\cite{BW2008,Min2010, Min2011,PN19a,Tran2019}. To be more specific about this technique, one establishes a level-set inequality involving fractional maximal operators and then, it allows us to control the level sets of fractional maximal function of gradient of solutions with the level sets of that of data. The idea of using WFMDs is naturally posed to recover integrability information of solutions by that of the given data and it valids for a wide range of rearrangement invariant quasi-normed spaces. Therefore, it is shown that the approach with WFMDs is general enough to be applicable to a larger class of nonlinear equations/systems, see~\cite{Min2010,NP21a} for further reading.

In the sequel, we shall state the theorem of estimates on level sets via WFMDs, Theorem~\ref{theo:dist}. It plays a key tool to establish regularity results in generalized weighted function spaces: Theorem~\ref{theo:main} and~\ref{theo:improv}, as we shall see in the following pages.

Second topic of the present paper, and probably more interestingly: we turn our attention to the generalized Stokes problems with asymptotically regular operators. We carry on the investigation about the regularity estimates for weak solutions to \emph{asymptotically regular} problem of the type
\begin{align}\label{eq:asym}
\begin{cases}
-\mathrm{div} \left(\mathbf{a}(x,\nabla \mathbf{u})\right) + \nabla \mathrm{\pi} &= -\mathrm{div}\left(\mathbf{B}(x,\mathbf{f})\right) \quad \, \text{in}\ \Omega, \\
 \hspace{1.5cm} \mathrm{div} (\mathbf{u}) &= \ 0  \hspace{2.45cm} \ \text{in}\ \Omega, \\
 \hspace{1.9cm} \mathbf{u} & = \ \mathbf{g} \hspace{2.55cm} \text{on}\ \partial\Omega.
 \end{cases}
\end{align} 
Here, we are interested in $\mathbf{a}: \Omega \times \mathbb{R}^{d^2} \to \mathbb{R}^{d^2}$, a Carath\'edory vector-valued function, that is so-called asymptotically $\delta_0$-regular with operator $\mathbf{A}$ in~\eqref{eq:Stokes} in the sense that there exists $\delta_0>0$ small enough such that 
\begin{align*}
\limsup_{|\nu| \to \infty}\frac{|\mathbf{a}(x,\nu)-\mathbf{A}(x,\nu)|}{(\sigma^2+|\nu|^2)^{\frac{p-1}{2}}} \le \delta_0,
\end{align*}
uniformly with respect to $x \in \Omega$. In this case, it allows $\mathbf{a}(x,\nu)$ to get close to regular operator $\mathbf{A}(x,\nu)$ as $|\nu|$ tends to infinity. To be more specific, the notion of \emph{asymtotically $\delta_0$-regular problem} will be captured in Definition~\ref{def:asymp} below. The notion of ``asymptotically regular operator'' was introduced in pioneristic paper~\cite{CE1986} by Chipot and Evans when they studied Lipschitz regularity for minimizers of asymptotically convex integrals and $p=2$. Later, Raymond obtained results for asymptotically convex functionals with $p$-growth in~\cite{Raymond1991}; and Foss~\cite{Foss2008} extended these Lipschitz estimates up to the boundary. Recently, Calder\'on-Zygmund theory and partial Lipschitz regularity for asymptotically regular elliptic systems and minimizers have been established separately by Scheven and Schmidt in~\cite{SS2010_I,SS2009_II}. Afterwards, as far as the asymptotic structure is concerned, higher integrability results for solutions to parabolic systems of $p$-Laplacian type have been investigated in~\cite{KM2012,Isernia2015}. More recent contributions include the works~\cite{BOW2015,BCO2015}, where the authors presented global results for asymptotically regular elliptic and parabolic problems in Reifenberg flat domain. The second aim of this paper is to address further the asymptotic regularity for generalized stationary Stokes systems~\eqref{eq:asym}. Notably, as in the previous topic, global gradient bounds via fractional maximal operators $\mathcal{M}_\alpha$ are also studied in the setting of generalized Lorentz and $\psi$-Morrey spaces. As far as we know, not much has been known regularity results concerning the asymptotically regular models regarding the Stokes systems. Though problem in mind is related to the regular one in~\eqref{eq:Stokes}, we shall prove them as stand alone results in this paper.

In this connection, as an application of the first topic in this paper,  theorem of asymptotic regularity will be derived as a consequence of global gradient bounds achieved in Theorems~\ref{theo:main} and~\ref{theo:improv}, as we are going to see in Theorem~\ref{theo:asym} described below.

The rest of this section is devoted to clarify some notation, definitions, assumptions and a few conventions that used throughout this paper. At the same time, we are also ready to state the main results of this paper (formal version); see below for the precise statements.

\subsection{Notation and Definitions}
\label{sec:notationanddef}

In what follows, constants are generically denoted by $C$ whose values are larger or equal than 1 and need not be the same at each point of use; when needed, peculiar dependence on parameters will be indicated using parenthesis. In addition, $B_R(x_0)$ denotes the open ball in $\mathbb{R}^d$ with center $x_0$ and radius $R>0$; and we write $|\mathfrak{B}|$ in place of $d$-dimensional Lebesgue measure of a measurable subset $\mathfrak{B} \subset \mathbb{R}^d$. Moreover, if $|\mathfrak{B}|$ is positive, the integral average of $\mathsf{g} \in L^1_{\mathrm{loc}}(\mathbb{R}^d)$ over $\mathfrak{B} \subset \mathbb{R}^d$ is defined by
$$
\overline{\mathsf{g}}_{\mathfrak{B}} = \fint_{\mathfrak{B}}{\mathsf{g}(x)dx} := \frac{1}{|\mathfrak{B}|} \int_{\mathfrak{B}}{\mathsf{g}(x)dx}.
$$
With a slight abuse of notation, we shall adopt $\{|\mathsf{g}|>\lambda\}$ instead of $\{x \in \Omega: |\mathsf{g}(x)| > \lambda\}$. 

There are some standing assumptions on regularity of the domain and of data.  For instance, domain $\Omega$ we consider in the problem may be nonsmooth, but in order to prove estimates up-to-boundary, $\partial\Omega$ is required to be flat in some sense. Throughout our study, to work with generalized Stokes system~\eqref{eq:Stokes}, we assume that $\Omega$ is a Reifenberg flat domain. Further, combined with the flatness of $\partial\Omega$, the nonlinearity coefficient $\mathbf{A}$ is also assumed to satisfy small BMO (bounded mean oscillation) condition. We shall shed on the light of these both assumptions, given below.

\begin{taggedassumption}{$\mathbf{H_1}$}
\label{ass:Reif}
The domain $\Omega$ is considered to be an $(\delta_0,R_0)$-Reifenberg flat domain for a given $R_0>0$ and a small positive constant $\delta_0>0$. That means, for any $x \in \partial \Omega$ and $\varrho \in (0, (1-\delta_0)R_0)$, there exists a coordinates system $\{\tilde{x}_1, \tilde{x}_2,\cdots,\tilde{x}_d\}$ such that the new origin $O \in \Omega$, $x = - \delta_0 \varrho/(1-\delta_0)\tilde{x}_d$ and
\begin{align}\label{cond:Reif}
B_{\varrho}(O) \cap \{\tilde{x}_d > 0\} \subset B_{\varrho}(O) \cap \Omega \subset B_{\varrho}(O) \cap \{\tilde{x}_d > -2 \delta_0 \varrho/(1-\delta_0) \}.
\end{align}
Here the notation $\{\tilde{x}_d > \tau\}$ presents to the set $\{\tilde{x} = (\tilde{x}_1, \tilde{x}_2,\cdots, \tilde{x}_d): \ \tilde{x}_d > \tau\}$.
\end{taggedassumption}

\begin{remark}
\label{rem:H1}
Assumption~\ref{ass:Reif} requires a geometric condition for domain $\Omega$. It prescribes that at all scales, the boundary  $\partial\Omega$ can be trapped between two hyperplanes, depending on the chosen scale. Furthermore, in Reifenberg flatness condition~\eqref{cond:Reif}, the parameter $\delta_0>0$ is a small constant and invariant under scaling.  However, this  geometric condition is meaningful when $0<\delta_0<1/8$ and the upper bound $1/8$ is chosen small enough to encode the idea of flatness at infinity. Readers may consult~\cite{Toro}, a very interesting source of Reifenberg flatness, for a more detailed exposition. The small value of $\delta_0$ will allow us to measure the deviation of $\partial\Omega$ at each scale $\varrho>0$. In this paper, the choice of $\delta_0$ will be specified later so that~\eqref{ineq-main} and~\eqref{GF-norm} in our main results hold true for all $\mathfrak{s} \in (0,\infty)$ (see Theorem~\ref{theo:main} and Theorem~\ref{theo:improv}). On the other hand, $R_0$ can be any number whose value is larger than 1 by the scaling invariance of the problem (see Lemma~\ref{lem:boundary}). Here, it also remarks that by scaling argument, $\delta_0>0$ is independent of $R_0$ for $R_0>1$.
\end{remark}

\begin{taggedassumption}{$\mathbf{H_2}$}
\label{ass:bmo}
For the couple $(\delta_0,R_0)$ specified as above, the matrix-valued function $\mathbf{A}$ satisfies the following condition
\begin{align}\label{BMO}
[\mathbf{A}]_{R_0} := \sup_{y \in \mathbb{R}^d, \ 0<\varrho\le R_0} \fint_{B_{\varrho}(y)} \left( \sup_{\nu \in \mathbb{R}^{d^2} \setminus \{0\}} \frac{|\mathbf{A}(x,\nu) - \overline{\mathbf{A}}_{B_{\varrho}(y)}(\nu)|}{(\sigma^2 +|\nu|^2)^{\frac{p-1}{2}}} \right) dx \le \delta_0,
\end{align}
where $\overline{\mathbf{A}}_{\mathfrak{B}}(\nu)$ stands for the average of $\mathbf{A}(\cdot,\nu)$ over the measurable set $\mathfrak{B}$.
\end{taggedassumption}

As aforementioned, these two assumptions will always be coupled in our work. And for the convenience of the reader, when both conditions~\ref{ass:Reif} and~\ref{ass:bmo} are imposed to domain $\Omega$ and the leading coefficients $\mathbf{A}$, we will often write $(\Omega,\mathbf{A}) \in \mathcal{H}(\delta_0,R_0)$. After remarkable works in~\cite{BW2004,BW2008}, there have been a series of developments leading to a set of results regarding global regularity estimates for boundary value problems under this geometric feature of $\Omega$ and hypothesis of $\mathbf{A}$. Remarkably, Reifenberg-flatness is a natural and minimal setting on geometrical structure of domain to prove boundary regularity theory for PDEs, among various concepts such as non-tangentially accessible (NTA) domain, uniform and John domains.  Here, we employ the use of specific assumption: $(\Omega,\mathbf{A}) \in \mathcal{H}(\delta_0,R_0)$ to generalized problem~\eqref{eq:Stokes} in conducting regularity results on the whole domain. Further, for notational purpose, we regard ``$\mathtt{data}_0$'' for the set of prescribed parameters of dependence. For instance,
\begin{align*}
\mathtt{data}_0 \equiv \mathtt{data}_0(d,p,\sigma,\Upsilon,\alpha, \Omega, R_0),
\end{align*}
so that when we write $C(\mathtt{data}_0)$ it means $C$ depends upon $\mathtt{data}_0$. 

Fractional maximal operators and distribution functions of them are the key tools of our argument. These are classical operators of interest in harmonic analysis and the theory of partial differential equations,~\cite{Grafakos}. As we shall see, in this study, we derive some new regularity results for solutions via fractional maximal functions. So, what is our purpose working with $\mathcal{M}_\alpha$? - It is known that, the oscillation of a function in Sobolev spaces can be controlled in terms of the fractional maximal function of its gradient. Therefore, via $\mathcal{M}_\alpha$, it allows us to derive both the size bounds for solutions and/or their derivatives, especially the fractional derivatives of order $\alpha$ (properties of solutions in fractional Sobolev spaces). 

Let us now recall here the definition for the reader's convenience. 
\begin{definition}[Fractional maximal operators]\label{def:Malpha}
Given $\alpha \in [0, d]$, the fractional maximal operator $\mathcal{M}_\alpha$ is defined by
\begin{align}\notag 
\mathcal{M}_\alpha \mathsf{g}(x) = \sup_{\varrho>0}{\varrho^\alpha \fint_{B_\varrho(x)}{|\mathsf{g}(z)|dz}}, 
\end{align}
for $x \in \mathbb{R}^d$ and $\mathsf{g} \in L^1_{\mathrm{loc}}(\mathbb{R}^d)$. Notably, as $\alpha=0$, it coincides with the Hardy-Littlewood operator, i.e. $\mathcal{M}_0 \equiv \mathcal{M}$.
\end{definition}

We also include here one of the most important piece of our interest that regarding $\mathcal{M}_\alpha$: the boundedness property, given in the following lemma. Obviously, such the result allows us to conclude in case $\alpha=0$. See ~\cite[Lemma 3.3]{NP21a} for detailed proof.
\begin{lemma}\label{bound-M-beta}
For every $\alpha \in [0, d)$, one can find a constant $C=C(d,\alpha)>0$ such that  the following inequality
\begin{align*}
|\{x \in \mathbb{R}^d: \ \mathcal{M}_\alpha \mathsf{g}(x)>\lambda\}| \le  C \left(\frac{1}{\lambda}\int_{\mathbb{R}^d}|\mathsf{g}(z)| dz\right)^{\frac{d}{d-\alpha}}
\end{align*}
holds for all $\lambda>0$ and $\mathsf{g} \in L^{1}(\mathbb{R}^d)$.
\end{lemma} 

\subsection{A key point estimate: level-set inequality}

In order to state main results in this paper, we organize to specific subsections. The first one forms a key tool to prove regularity results, in which we establish a weighted distribution inequality on level sets of $\mathcal{M}_\alpha$ of the gradient of solutions and the assigned data, associated to a Muckenhoupt weight $\omega \in \mathcal{A}_{\infty}$. This result is based on the idea of \emph{weighted fractional maximal distribution functions}, that we are going to detail in a few lines below.

\begin{definition}[$\mathcal{A}_p$ classes and Muckenhoupt weights]\label{def:Muck}
Let $1\le q < \infty$ and  $\omega \in L^1_{\mathrm{loc}}(\mathbb{R}^d)$ be a non-negative function. We say that $\omega$ belongs to the $\mathcal{A}_q$ class if $[\omega]_{\mathcal{A}_q}<\infty$, where
\begin{align*}
[\omega]_{\mathcal{A}_q} := \begin{cases} \displaystyle{\sup_{B \subset \mathbb{R}^d} \left(\fint_{B} \omega(z) dz\right) \sup_{z \in B} [\omega(z)]^{-1}}, & \quad \mbox{ if } q=1, \\
 \displaystyle{\sup_{B \subset \mathbb{R}^d} \left(\fint_{B} \omega(z) dz\right)\left(\fint_{B} [\omega(z)]^{-\frac{1}{q-1}}dz\right)^{q-1}}, & \quad \mbox{ if } q>1,
 \end{cases}
\end{align*}
for any ball $B =B_r(x)$ in $\mathbb{R}^d$. 
\end{definition}
As in~\cite{BS2017}, it is well known that when $\omega \in \mathcal{A}_{\infty}$, there are positive constants $C_1,C_2$ and $\iota_1,\iota_2$ such that
\begin{align}\label{ineq-Muck}
C_1 \left(\frac{|E|}{|B|}\right)^{\iota_1} {\omega(B)} \le {\omega(E)} \le C_2 \left(\frac{|E|}{|B|}\right)^{\iota_2} {\omega(B)},
\end{align}
for all measurable subset $E$ of the ball $B$, where $\omega(E) = \int_E \omega(z) dz$. Then, we shall simply write $[\omega]_{\mathcal{A}_\infty} = (C_1,C_2,\iota_1,\iota_2)$ and for the rest of paper when dealing with Muckenhoupt weights, we further define 
\begin{align}\notag
\mathtt{data}_1 := (\mathtt{data}_0,[\omega]_{\mathcal{A}_\infty}).
\end{align}

\begin{definition}[Weighted fractional distribution functions (WFMDs)]
\label{def:wfmds}
Let $\omega \in \mathcal{A}_{\infty}$ and $\mathsf{g} \in L^1_{\mathrm{loc}}(\mathbb{R}^d)$. The weighted distribution function $\mathbf{d}^{\omega}_{\mathsf{g}}: \, \mathbb{R}^+ \to \mathbb{R}^+$ is defined as follows:
\begin{align}\notag 
\mathbf{d}^{\omega}_{\mathsf{g}}(\lambda) & := \omega\left(\left\{x \in \Omega: \ |\mathsf{g}(x)|> \lambda\right\}\right), \quad \text{for all } \ 0 \le \lambda <\infty.
\end{align}
Given $\alpha \in [0,d)$, the weighted fractional maximal distribution function (WFMD), denoted by $\mathbf{D}^{\alpha,\omega}_{\mathsf{g}}: \, \mathbb{R}^+ \to \mathbb{R}^+$ as
\begin{align}\label{def-Df}
\mathbf{D}^{\alpha,\omega}_{\mathsf{g}}(\lambda) & := \mathbf{d}^{\omega}_{\mathcal{M}_{\alpha}\mathsf{g}}(\lambda),  \quad \text{for all } \ 0 \le \lambda <\infty.
\end{align}
\end{definition}

With a view to avoid repetition and also for the sake of brevity, in what follows we shall initially assume that $(\mathbf{u},\mathrm{\pi}) \in W^{1,p}_{\mathbf{g},\mathrm{div}}(\Omega;\mathbb{R}^d) \times L^{p'}_{\mathrm{int}}(\Omega)$ is a weak solution pair to system~\eqref{eq:Stokes} under the assumptions~\eqref{cond:A1}-\eqref{cond:Bf}, with given data $\mathbf{f} \in L^p(\Omega;\mathbb{R}^{d^2})$ and $\mathbf{g} \in W^{1,p}(\Omega;\mathbb{R}^{d})$. In the context, we moreover define
\begin{align}\notag 
\mathbb{U}(x) := |\nabla \mathbf{u}|^p + |\mathrm{\pi}|^{p'}, \quad \mathbb{F}_{\sigma}(x) := \sigma^p + |\mathbf{f}|^p + |\nabla \mathbf{g}|^p, \quad x \in \Omega.
\end{align}
and use the abbreviations $\mathbb{X}, \mathbb{Y}: [0,\infty) \to [0,\infty)$ for the two distribution functions as following
\begin{align}\label{def:XY}
\mathbb{X}(\lambda) := \mathbf{D}^{\alpha,\omega}_{\mathbb{U}} (\lambda), \quad \mathbb{Y}(\lambda) := \mathbf{D}^{\alpha,\omega}_{\mathbb{F}_{\sigma}} (\lambda), \quad \lambda \ge 0.
\end{align}

Now, we are ready to state the first main result of this work: level-set inequality on the idea of WFMDs, a technical tool in proving our regularity theorems.
\begin{theorem}\label{theo:dist}
There is a constant $\mathfrak{a} = \mathfrak{a}(\mathtt{data}_1)>0$ such that for every $\varepsilon \in (0,1)$, one can find $\delta_0=\delta_0(\varepsilon,\mathtt{data}_1) \in (0,1/81)$, $\mathfrak{b} = \mathfrak{b}(\varepsilon,\mathtt{data}_1)>0$ and $C=C(\mathtt{data}_1)>0$ satisfying
\begin{align}\label{ineq-dist}
\mathbb{X}(\mathfrak{a}\lambda) \le C \left[\varepsilon \mathbb{X}(\lambda) +  \mathbb{Y}(\mathfrak{b}\lambda)\right], 
\end{align}
if provided an additional assumption $(\Omega,\mathbf{A}) \in \mathcal{H}(\delta_0,R_0)$ for some $R_0>0$. 
\end{theorem}

This theorem has it own original version studied by Acerbi and Mingione in~\cite{AM2007} in the setting of parabolic equations. They initiated an iteration-covering technique, that is completely harmonic analysis free (for instance the good-$\lambda$-inequality) and avoids the use of the maximal function operator, only argued on certain Calder\'on-Zygmund-type covering lemma. Here, we exploit the level-set argument and distribution functions acting on fractional maximal operators to obtain a series of regularity results for gradient of solutions via $\mathcal{M}_\alpha$. The plan to prove this theorem is based on the use of reverse H\"older inequality, comparison scheme with suitable homogeneous problems and a standard covering argument. Specially, the structural assumptions on $\Omega$ and $\mathbf{A}$: $(\Omega,\mathbf{A}) \in \mathcal{H}(\delta_0,R_0)$ are imposed to carry results up to the boundary. For the proof we refer to Section~\ref{sec:proof}.

\subsection{Regularity estimates in generalized spaces}

As already alluded above, we shall focus our attention on the generalized weighted Lorentz and $\psi$-generalized Morrey spaces. Let us now be a bit more precise and describe how these spaces are defined.  First, the key feature of our approach is that we address global bounds on generalized Lorentz spaces with two mixed weights. Here, exploiting an idea given in~\cite{CRS07}, we consider a new weight $\mu \in L^1_{\mathrm{loc}}(\mathbb{R}^+;\mathbb{R}^+)$ and a non-decreasing function $\Psi$ given by
\begin{align}\label{def:V}
\Psi(r) = \int_0^r \mu(s) ds, \quad r \in [0,\infty),
\end{align}
which allows us to formulate the definition of a space with two weight functions $\omega \in \mathcal{A}_{\infty}$ and $\mu$ taken into account, as follows.
\begin{definition}[Generalized weighted Lorentz spaces]\label{def:Lorentz}
Let $\omega \in \mathcal{A}_{\infty}$ and $\mu$ be two weight functions defined in the above-mentioned. Given $\mathsf{g}$ a measurable function on $\Omega$ and $\mathfrak{s} \in (0,\infty)$, $0 <\mathfrak{t}\le \infty$, define
\begin{align*}
\|\mathsf{g}\|_{\mathcal{L}^{\mathfrak{s},\mathfrak{t}}_{\mu,\omega}(\Omega)} := \begin{cases} \left[\mathfrak{s} \displaystyle{\int_0^\infty \lambda^{\mathfrak{t}} \left[\Psi\left(\mathbf{d}^{\omega}_{\mathsf{g}}(\lambda)\right) \right]^{\frac{\mathfrak{t}}{\mathfrak{s}}} \frac{d\lambda}{\lambda}} \right]^{\frac{1}{\mathfrak{t}}},
\ & \mbox{ if } \mathfrak{t}<\infty,\\ 
\displaystyle{\sup_{\lambda>0} \left\{\lambda \left[\Psi\left(\mathbf{d}^{\omega}_{\mathsf{g}}(\lambda)\right)\right]^{\frac{1}{\mathfrak{s}}}\right\}},
\ & \mbox{ if } \mathfrak{t} = \infty. \end{cases}
\end{align*}
Then, the set of all functions $\mathsf{g}$ with $\|\mathsf{g}\|_{\mathcal{L}^{\mathfrak{s},\mathfrak{t}}_{\mu,\omega}(\Omega)}<\infty$, denoted by $\mathcal{L}^{s,t}_{\mu,\omega}(\Omega)$, is called the generalized weighted Lorentz spaces with indices $\mathfrak{s},\mathfrak{t}$ and two imposed weights $\mu,\omega$. 
\end{definition}
The following remarks are related to a few comments from Definition~\ref{def:Lorentz}.
\begin{remark} $ $
\begin{itemize}
\item  When $\omega \equiv \mu \equiv  1$, then $\mathcal{L}^{\mathfrak{s},\mathfrak{t}}_{\mu,\omega}(\Omega)$ is the classical Lorentz space $L^{\mathfrak{s},\mathfrak{t}}(\Omega)$. Even further, $L^{\mathfrak{s},\mathfrak{s}}(\Omega)$ is exactly the Lebesgue space $L^{\mathfrak{s}}(\Omega)$.
\item With $\mathfrak{t}=\infty$, $\mathcal{L}^{\mathfrak{s},\infty}_{\mu,\omega}(\Omega)$ is also called the generalized weighted Marcinkiewicz space.
\item The mapping $\|\cdot\|_{\mathcal{L}^{\mathfrak{s},\mathfrak{t}}_{\mu,\omega}(\Omega)}$ is a quasi-norm if and only if $\Psi$ belongs to $\Delta_2$-class, which means that there exists $\beta_2>0$ satisfying
\begin{align}\label{Del-2}
\Psi(2\lambda) \le \beta_2 \Psi(\lambda), \quad \mbox{ for every } \lambda \ge 0.
\end{align}
\end{itemize}
\end{remark}

Our main regularity result concerning generalized Lorentz spaces is the following.

\begin{theorem}\label{theo:main}
Assume that $\mu \in L^1_{\mathrm{loc}}(\mathbb{R}^+; \mathbb{R}^+)$ and $\Psi$ is given as in~\eqref{def:V} such that
\begin{align}\label{cond:V}
\beta_1 \Psi(\lambda) \le \Psi(2\lambda) \le \beta_2 \Psi(\lambda), \quad \forall \lambda \ge 0,
\end{align}
for some $\beta_2 > \beta_1 > 1$. For every $\mathfrak{s} \in (0,\infty)$ and $0 < \mathfrak{t} \le \infty$, there exists $\delta_0 = \delta_0(\mathtt{data}_2) \in (0,1/81)$ such that if $(\Omega,\mathbf{A}) \in \mathcal{H}(\delta_0,R_0)$ for some $R_0>0$, then the estimate
\begin{align}\label{ineq-main}
\|\mathcal{M}_{\alpha}(|\nabla \mathbf{u}|^p + |\mathrm{\pi}|^{p'})\|_{\mathcal{L}^{\mathfrak{s},\mathfrak{t}}_{\mu,\omega}(\Omega)} \le C \|\mathcal{M}_{\alpha}(\sigma^p + |\mathbf{f}|^p + |\nabla \mathbf{g}|^p)\|_{\mathcal{L}^{\mathfrak{s},\mathfrak{t}}_{\mu,\omega}(\Omega)}
\end{align}
holds true with a constant $C$ depending on $\mathtt{data}_2 := (\mathtt{data}_1,\mathfrak{s},\mathfrak{t},\beta_1,\beta_2)$.
\end{theorem}

\begin{definition}[$\psi$-generalized Morrey spaces]
Let $\psi: \Omega \times \mathbb{R}^+ \to \mathbb{R}^+$ be a measurable function. Then, a measurable map $\mathsf{g}: \Omega \to \mathbb{R}$ is said to belong to the $\psi$-generalized Morrey space $\mathrm{M}^{\mathfrak{s},\psi}(\Omega)$ with $\mathfrak{s} \in (0, \infty)$ if
\begin{align}\label{Morrey-norm}
\|\mathsf{g}\|_{\mathrm{M}^{\mathfrak{s},\psi}(\Omega)} := \sup_{y\in \Omega; \, 0<\varrho<\mathrm{diam}(\Omega)} \left(\frac{1}{\psi(x,\varrho)}\int_{\Omega_{\varrho}(y)}|\mathsf{g}(z)|^{\mathfrak{s}}dz\right)^{\frac{1}{\mathfrak{s}}}< \infty.
\end{align}
\end{definition}

As a consequence of the previous result, the second regularity result reveals the gradient bound of weak solution pair in $\psi$-generalized Morrey spaces. Let us report in Theorem~\ref{theo:improv} below.
 
\begin{theorem}\label{theo:improv}
Let $\psi: \Omega \times \mathbb{R}^+ \to \mathbb{R}^+$ be a measurable function such that there is $\beta_0 \in (1,2^d)$ satisfying 
\begin{align}\label{cond-psi-2}
\psi(x,2\varrho) \le \beta_0 \psi(x,\varrho), \quad \forall x \in \Omega \mbox{ and } 0 < \varrho <\mathrm{diam}(\Omega).
\end{align}
Then, for every $\mathfrak{s} \in (0,\infty)$, there exists $\delta_0 = \delta_0(\mathtt{data}_3) \in (0,1/81)$ such that if $(\Omega,\mathbf{A}) \in \mathcal{H}(\delta_0,R_0)$ for some $R_0>0$, it holds 
\begin{align}\label{GF-norm}
\|\mathcal{M}_{\alpha}(|\nabla \mathbf{u}|^p + |\mathrm{\pi}|^{p'})\|_{\mathrm{M}^{\mathfrak{s},\psi}(\Omega)} \le C \|\mathcal{M}_{\alpha}(\sigma^p + |\mathbf{f}|^p + |\nabla \mathbf{g}|^p)\|_{\mathrm{M}^{\mathfrak{s},\psi}(\Omega)},
\end{align}
where the constant $C$ depends on $\mathtt{data}_3 := (\mathtt{data}_0,\mathfrak{s},\beta_0)$.
\end{theorem}

\subsection{A correlation to asymptotically regular problems}

The next result is devoted to study in details the regularity estimates for asymptotically regular Stokes system~\ref{eq:asym}. Let us present in the following the definition of so-called \emph{asymptotically $\delta_0>0$-regular problem}. For additional reading, one can refer to~\cite[Theorem 2.14]{SS2010_I} and~\cite[Definition 2.2]{BOW2015} for some of precise statements.

\begin{definition}[Asymptotically $\delta_0$-regular problem]
\label{def:asymp}
Let $\mathbf{A}(x,\nu)$ be a regular function that satisfying~\eqref{cond:A1} and~\eqref{cond:A2}. Then, we say that $\mathbf{a}(x,\nu)$ is asymptotically $\delta_0$-regular with $\mathbf{A}(x,\nu)$ if there exists a uniformly bounded nonnegative function $\epsilon: [0,\infty) \to [0,\infty)$ such that
\begin{align*}
\displaystyle{\limsup_{\lambda \to\infty}{\epsilon(\lambda)}} \le \delta_0,
\end{align*}
and
\begin{align*}
|\mathbf{a}(x,\nu) - \mathbf{A}(x,\nu)| \le \epsilon(|\nu|)(1+|\nu|^{p-1}),
\end{align*}
for almost every $x \in \mathbb{R}^d$ and all $\nu \in \mathbb{R}^{d^2}$. 
\end{definition}
\begin{remark}
\label{rem:asymp}
For any sufficiently small $\delta_0>0$, if $\mathbf{a}(x,\nu)$ is asymptotically $\delta_0$-regular with $\mathbf{A}(x,\nu)$, then one has that
\begin{align}
\label{cond:asym}
\limsup_{|\nu| \to \infty}{\frac{|\mathbf{a}(x,\nu)-\mathbf{A}(x,\nu)|}{(\sigma^2+|\nu|^2)^{\frac{p-1}{2}}}} \le \delta_0
\end{align}
uniformly with respect to $x \in \mathbb{R}^d$. It remarks that, as mentioned in~\cite{BOW2015}, the asymptotically $\delta_0$-regular condition given in~\eqref{cond:asym} is weaker than the one described in~\cite{SS2010_I}. Throughout this paper we will assume that $\mathbf{a}(x,\nu)$ is asymptotically $\delta_0$-regular with operator $\mathbf{A}(x,\nu)$ in problem~\eqref{eq:Stokes} in the sense of~\eqref{cond:asym}, where $\delta_0$ will be selected later on.
\end{remark}

Directly follow from previous results, it allows us to state the global gradient boundedness for solution pair to asymptotically regular systems of Stokes type in~\eqref{eq:asym} that settled in corresponding generalized function spaces. The next main result of the paper reads as follows.

\begin{theorem}\label{theo:asym}
Let $0<\mathfrak{s}<\infty$, $0<\mathfrak{t}\le \infty$ and $\mu,\Psi,\psi$ be given as in~\eqref{cond:V},~\eqref{cond-psi-2}. Suppose that $\mathbf{a}: \Omega \times \mathbb{R}^{d^2} \to \mathbb{R}^{d^2}$ is a Carath\'edory matrix-valued function. Assume moreover that $(\mathbf{u}_{\mathbf{a}},\mathrm{\pi}_{\mathbf{a}}) \in W^{1,p}_{\mathbf{g},\mathrm{div}}(\Omega;\mathbb{R}^d) \times L^{p'}_{\mathrm{int}}(\Omega)$ is a pair of solutions to system~\eqref{eq:asym}. There exists $\delta_0 \in (0,1/81)$ such that if $(\Omega,\mathbf{A}) \in \mathcal{H}(\delta_0,R_0)$ for some $R_0>0$ and~\eqref{cond:asym} is valid, then the following estimates
\begin{align}\label{main-asym}
\|\mathcal{M}_{\alpha}(|\nabla \mathbf{u}_{\mathbf{a}}|^p + |\mathrm{\pi}_{\mathbf{a}}|^{p'})\|_{\mathbb{L}(\Omega)} \le C \|\mathcal{M}_{\alpha}(1 + |\mathbf{f}|^p + |\nabla \mathbf{g}|^p)\|_{\mathbb{L}(\Omega)},
\end{align}
holds. Here, for the sake of shortness, we write $\mathbb{L}(\Omega) = \mathcal{L}^{\mathfrak{s},\mathfrak{t}}_{\mu,\omega}(\Omega)$ or $\mathbb{L}(\Omega) = \mathrm{M}^{\mathfrak{s},\psi}(\Omega)$ corresponding to the generalized weighted Lorentz or $\psi$-generalized Morrey spaces.
\end{theorem}

The rest part of this paper is arranged as follows. The core of Section~\ref{sec:Exist} is to prove the global a priori estimate of weak solution to our problem. Section~\ref{sec:comp} is intended to establish the comparison procedures, which play a technical keypoint in our argument. The comparison scheme is divided into two parts: in the interior and up-to-the-boundary of domain $\Omega$. Next, the main proofs of this paper is split into two parts: first, to deal with regularity estimates for generalized Stokes system, we send the reader to Section~\ref{sec:proof} that proving main theorems stated above: Theorem~\ref{theo:main} and Theorem~\ref{theo:improv}. Finally, being connected with the results of~\eqref{eq:Stokes}, it enables us to  prove desired regularity for asymptotically regular system~\eqref{eq:asym}, already stated in Theorem~\ref{theo:asym}.

\section{Global a priori estimate for weak solution}
\label{sec:Exist}

This section aims to state and prove a theorem that enables us to establish a priori estimate for weak solution pair $(\mathbf{u},\pi)$ to steady problem~\eqref{eq:Stokes}. Let us first reproduce here the solvability result of divergence equations presented by Acosta \textit{et al.} in~\cite{ADM06}. Regarding that study, via a constructive approach, existence of solutions to the divergence equation on a bounded John domain was proved. 

\begin{lemma}\label{lem:div}
Let $\Omega$ be an open bounded John domain in $\mathbb{R}^d$ and $h \in L^q_{\mathrm{int}}(\Omega)$ for some $q>1$. Then there exists $\varphi \in W_0^{1,q}(\Omega;\mathbb{R}^d)$ and $C = C(\Omega, d, q)>0$ such that
\begin{align}\notag 
\mathrm{div}(\varphi) = h \ \mbox{ in } \Omega, \mbox{ and } \|\nabla \varphi\|_{L^q(\Omega;\mathrm{R}^d)} \le C \|h\|_{L^q(\Omega)}.
\end{align}
\end{lemma}
\begin{remark}
It should also be noted that in the existence proof pointed out in Lemma~\ref{lem:div}, domain to be considered is the John domain. However, as shown in~\cite{LM2014, DK2018_2}, any domain that is flat in the sense of Reifenberg is also a John domain. We will make use of this fact and in what follows, the only assumption $(\Omega,\mathbf{A}) \in \mathcal{H}(\delta_0,R_0)$ is imposed to our problems.
\end{remark}
\begin{lemma}\label{lem:global}
Let $\mathbf{f} \in L^p(\Omega;\mathbb{R}^{d^2})$ and $\mathbf{g} \in W^{1,p}(\Omega;\mathbb{R}^{d})$. There exists a unique weak solution pair $(\mathbf{u},\mathrm{\pi}) \in W^{1,p}_{\mathbf{g},\mathrm{div}}(\Omega;\mathbb{R}^d) \times L^{p'}_{\mathrm{int}}(\Omega)$ to system~\eqref{eq:Stokes} such that
\begin{align}\label{est:global-1}
\|\nabla \mathbf{u}\|_{L^p(\Omega;\mathbb{R}^{d^2})} + \|\mathrm{\pi}\|_{L^{p'}(\Omega)} \le C \left(\sigma + \|\mathbf{f}\|_{L^p(\Omega;\mathbb{R}^{d^2})} + \|\nabla \mathbf{g}\|_{L^p(\Omega;\mathbb{R}^{d^2})}\right),
\end{align}
where $C = C(\Omega,d,\Upsilon,p)>0$. In particular, there holds
\begin{align}\label{est:global-2}
\|\mathrm{\pi}\|_{L^{p'}(\Omega)} \le C \left(\sigma + \|\nabla \mathbf{u}\|_{L^{p}(\Omega;\mathbb{R}^{d^2})} + \|\mathbf{f}\|_{L^{p}(\Omega;\mathbb{R}^{d^2})}\right).
\end{align}
\end{lemma}
\begin{proof}
The existence and uniqueness of weak solutions have been proved by appealing the theory of monotone operators combined with compactness arguments (cf.~\cite{Ladyzenskaya,Ladyzenskaya2,Ladyzenskaya1,Lions1969}). Let us start putting into~\eqref{var-form1} the test function $\varphi = \mathbf{u}-\mathbf{g}$, we infer that
\begin{align}\notag
\int_{\Omega} \langle \mathbf{A}(x, \nabla \mathbf{u}), \nabla \mathbf{u} \rangle dx =  \int_{\Omega} \langle \mathbf{A}(x, \nabla \mathbf{u}), \nabla  \mathbf{g} \rangle dx + \int_{\Omega} \langle \mathbf{B}(x,\mathbf{f}), \nabla \mathbf{u} \rangle dx - \int_{\Omega} \langle \mathbf{B}(x,\mathbf{f}), \nabla \mathbf{g} \rangle dx.
\end{align}
With the use of~\eqref{cond:A1},~\eqref{cond:A2},~\eqref{cond:Bf}, it allows us to get
\begin{align}\label{est:global-3}
\Upsilon^{-1}\int_{\Omega} (\sigma^2+|\nabla \mathbf{u}|^2)^{\frac{p-2}{2}}|\nabla \mathbf{u}|^2 dx & \le \Upsilon \int_{\Omega} \sigma^{p-1}|\nabla \mathbf{u}| dx + \Upsilon \int_{\Omega} (\sigma^2+|\nabla \mathbf{u}|^2)^{\frac{p-1}{2}}|\nabla \mathbf{g}| dx \notag \\
& \qquad + \Upsilon\int_{\Omega} (\sigma^2+|\mathbf{f}|^2)^{\frac{p-1}{2}}|\nabla \mathbf{u}| dx + \Upsilon\int_{\Omega} (\sigma^2+|\mathbf{f}|^2)^{\frac{p-1}{2}}| \nabla \mathbf{g}| dx.
\end{align}
At this stage, thanks to Young's inequality, for every $\epsilon>0$, there exists a constant depending on $p,\epsilon$ such that the following estimate holds
\begin{align}\notag 
|\nabla \mathbf{u}|^p \le  \epsilon \left(\sigma^2 + |\nabla \mathbf{u}|^2\right)^{\frac{p}{2}} + C (\sigma^2+|\nabla \mathbf{u}|^2)^{\frac{p-2}{2}}|\nabla \mathbf{u}|^2,
\end{align}
and from~\eqref{est:global-3}, this implies
\begin{align}\label{est:global-3b}
\int_{\Omega} |\nabla \mathbf{u}|^p dx & \le 3\epsilon \int_{\Omega} |\nabla \mathbf{u}|^p dx + C_{\epsilon} \int_{\Omega} (\sigma^p + |\mathbf{f}|^p + |\nabla \mathbf{g}|^p) dx.
\end{align}
By choosing a suitable value of $\epsilon>0$ in~\eqref{est:global-3b}, we deduce that
\begin{align}\label{est:global-4}
\|\nabla \mathbf{u}\|_{L^p(\Omega;\mathbb{R}^{d^2})} \le C \left(\sigma + \|\mathbf{f}\|_{L^p(\Omega;\mathbb{R}^{d^2})} + \|\nabla \mathbf{g}\|_{L^p(\Omega;\mathbb{R}^{d^2})}\right).
\end{align}
Next, let us consider the following functional
\begin{align}\notag
 \mathcal{F}(\varphi) := \int_{\Omega} \langle \mathbf{A}(x,\nabla \mathbf{u}) - \mathbf{B}(x,\mathbf{f}), \nabla \varphi\rangle dx, \quad \varphi \in W_0^{1,p}(\Omega; \mathbb{R}^d).
\end{align}
Thanks to~\eqref{var-form2}, we know that $\mathcal{F}(\varphi) = 0$ for all $\varphi \in W^{1,p}_{0,\mathrm{div}}(\Omega;\mathbb{R}^d)$. On the other hand, we are allowed to apply~\eqref{cond:A1},~\eqref{cond:Bf} and H\"older's inequality that for every $\varphi \in W_0^{1,p}(\Omega; \mathbb{R}^d)$, it yields that
\begin{align}\label{est-Fvarphi}
|\mathcal{F}(\varphi)| & \le \Upsilon \left(\int_{\Omega} (\sigma^2 + |\nabla \mathbf{u}|^2)^{\frac{p-1}{2}} |\nabla \varphi| dx + \int_{\Omega} (\sigma^2 + |\mathbf{f}|^2)^{\frac{p-1}{2}} |\nabla \varphi| dx\right) \notag \\
& \le C \left(\sigma + \|\mathbf{f}\|_{L^p(\Omega;\mathbb{R}^{d^2})} + \|\nabla \mathbf{u}\|_{L^p(\Omega;\mathbb{R}^{d^2})}\right) \|\nabla \varphi\|_{L^p(\Omega;\mathbb{R}^{d^2})}.
\end{align}
It follows from inequality~\eqref{est:global-4} that
\begin{align}\notag
|\mathcal{F}(\varphi)| & \le C \left(\sigma + \|\mathbf{f}\|_{L^p(\Omega;\mathbb{R}^{d^2})} + \|\nabla \mathbf{g}\|_{L^p(\Omega;\mathbb{R}^{d^2})}\right) \|\nabla \varphi\|_{L^p(\Omega;\mathbb{R}^{d^2})},
\end{align}
and we are able to conclude that $\mathcal{F}$ is a bounded linear functional on $W_0^{1,p}(\Omega; \mathbb{R}^d)$. At this point, thanks to~\cite[Theorem III.5.3]{G11}, there exists a uniquely determined $\mathrm{\pi} \in L^{p'}_{\mathrm{int}}(\Omega)$ such that
$$ \mathcal{F}(\varphi) = \int_{\Omega} \mathrm{\pi} \mathrm{div}(\varphi) dx,$$
for all $\varphi \in W_0^{1,p}(\Omega;\mathbb{R}^d)$. Finally, it concludes that $(\mathbf{u},\mathrm{\pi}) \in W^{1,p}_{\mathbf{g},\mathrm{div}}(\Omega;\mathbb{R}^d) \times L^{p'}_{\mathrm{int}}(\Omega)$ is the unique pair of solutions to~\eqref{eq:Stokes}.

In order to deal with~~\eqref{est:global-1} and~\eqref{est:global-2}, let us now consider 
\begin{align}\notag
h := |\mathrm{\pi}|^{p'-2}\mathrm{\pi} - \fint_{\Omega} |\mathrm{\pi}|^{p'-2}\mathrm{\pi} dx,
\end{align}
and it is easy to check that $h \in L^{p}_{\mathrm{int}}(\Omega)$ and $\|h\|_{L^{p}(\Omega)} \le \|\mathrm{\pi}\|_{L^{p'}(\Omega)}^{p'-1}$. Applying Lemma~\ref{lem:div}, there exists $\varphi_h \in W_0^{1,p}(\Omega;\mathbb{R}^d)$ and $C = C(\Omega, d, p)>0$ such that
\begin{align}\label{est:eq-div}
\mathrm{div}(\varphi_h) = h \ \mbox{ in } \Omega, \mbox{ and } \|\nabla \varphi_h\|_{L^{p}(\Omega;\mathrm{R}^d)} \le C \|h\|_{L^{p}(\Omega)}.
\end{align} 
Here, keeping in mind that $\int_{\Omega} \mathrm{\pi} dx = 0$, we therefore obtain
\begin{align}\notag
\int_{\Omega} |\mathrm{\pi}|^{p'} dx = \int_{\Omega} \mathrm{\pi} \left(|\mathrm{\pi}|^{p'-2}\mathrm{\pi} - \fint_{\Omega} |\mathrm{\pi}|^{p'-2}\mathrm{\pi} dx\right) dx = \int_{\Omega} \mathrm{\pi} h dx = \int_{\Omega} \mathrm{\pi} \mathrm{div}(\varphi_h) dx = \mathcal{F}(\varphi_h).
\end{align}
Inserting~\eqref{est-Fvarphi} and~\eqref{est:eq-div} into this equality, one gets that
\begin{align}
\|\mathrm{\pi}\|_{L^{p'}(\Omega)}^{p'} & \le C \left(\sigma + \|\nabla \mathbf{u}\|_{L^p(\Omega;\mathbb{R}^{d^2})} + \|\mathbf{f}\|_{L^p(\Omega;\mathbb{R}^{d^2})}\right) \|\nabla \varphi_h\|_{L^p(\Omega;\mathbb{R}^{d^2})} \notag \\
& \le  C \left(\sigma + \|\nabla \mathbf{u}\|_{L^p(\Omega;\mathbb{R}^{d^2})}+ \|\mathbf{f}\|_{L^p(\Omega;\mathbb{R}^{d^2})}\right) \|h\|_{L^{p}(\Omega)} \notag \\
& \le  C \left(\sigma + \|\nabla \mathbf{u}\|_{L^p(\Omega;\mathbb{R}^{d^2})}+ \|\mathbf{f}\|_{L^p(\Omega;\mathbb{R}^{d^2})}\right) \|\mathrm{\pi}\|_{L^{p'}(\Omega)}^{p'-1},\notag
\end{align}
and this completes the proof of~\eqref{est:global-2}.  Finally,~\eqref{est:global-1} holds true by combining~\eqref{est:global-4} and~\eqref{est:global-2}.
\end{proof}

\section{Comparison procedures}\label{sec:comp}

As already described, our general plan to prove Theorem~\ref{theo:dist} is to establish a series of comparison procedures with reference problems. We are going to separately study and prove suitable comparison estimates for solution $(\mathbf{u},\mathrm{\pi})$ to the original problem in several steps. Sections~\ref{sec:interior} and~\ref{sec:boundary} are devoted to present in details the comparison strategy in the interior and up to the boundary, respectively.

\subsection{Interior comparison scheme}
\label{sec:interior}

As the first step, we establish the comparison estimate between weak solution $(\mathbf{u},\mathrm{\pi})$ of the original inhomogeneous system and solution $(\mathbf{v},\mathrm{\pi}_{\mathbf{v}})$ of the associated homogeneous one. This is the content of the following lemma.

\begin{lemma}\label{lem:comp-1}
Let $(\mathbf{u},\mathrm{\pi}) \in W^{1,p}_{\mathbf{g},\mathrm{div}}(\Omega;\mathbb{R}^d) \times L^{p'}_{\mathrm{int}}(\Omega)$ be a weak solution pair of system~\eqref{eq:Stokes} with the given data $\mathbf{f} \in L^p(\Omega;\mathbb{R}^{d^2})$ and $\mathbf{g} \in W^{1,p}(\Omega;\mathbb{R}^{d})$. Let $x_0 \in \Omega$ and $R >0$ such that $\mathfrak{B}_4:= B_{4R}(x_0) \subset \Omega$, assume that $(\mathbf{v},\mathrm{\pi}_{\mathbf{v}}) \in W^{1,p}_{\mathbf{u}-\mathbf{g},\mathrm{div}}(\mathfrak{B}_4;\mathbb{R}^d) \times L^{p'}_{\mathrm{int}}(\mathfrak{B}_4)$ is the unique solution to the following problem
\begin{align}\label{eq:hom-1}
\begin{cases}
-\mathrm{div}(\mathbf{A}(x,\nabla \mathbf{v})) + \nabla \mathrm{\pi}_{\mathbf{v}} &= \ 0 \hspace{1.25cm} \ \, \text{in}\ \mathfrak{B}_4, \\
 \hspace{1.4cm} \mathrm{div} (\mathbf{v}) &= \ 0  \hspace{1.3cm} \ \text{in}\ \mathfrak{B}_4, \\
 \hspace{1.8cm} \mathbf{v} & = \ \mathbf{u}-\mathbf{g} \hspace{0.7cm} \text{on}\ \partial \mathfrak{B}_4.
 \end{cases}
\end{align}
Then, there exist two constants $C = C(d,\Upsilon,p)>0$ and $\overline{p}=\overline{p}(p)>0$ such that 
\begin{align}\label{est-comp-1}
 \fint_{\mathfrak{B}_4}|\nabla \mathbf{u} - \nabla \mathbf{v}|^p + |\mathrm{\pi} - \mathrm{\pi}_{\mathbf{v}}|^{p'} dx  & \le \varepsilon \fint_{\mathfrak{B}_4} \left(\sigma^2 + |\nabla \mathbf{u}|^2\right)^{\frac{p}{2}} dx \notag \\
& \hspace{1cm} + C \varepsilon^{-\overline{p}} \fint_{\mathfrak{B}_4} \left(\sigma^p + |\mathbf{f}|^p + |\nabla \mathbf{g}|^p\right) dx,
\end{align}
for every $\varepsilon \in (0,1)$. 
\end{lemma}
\begin{proof}
First, using the fact that $(\mathbf{u},\mathrm{\pi}) \in W^{1,p}_{\mathbf{g},\mathrm{div}}(\Omega;\mathbb{R}^d) \times L^{p'}_{\mathrm{int}}(\Omega)$ is a weak solution to~\eqref{eq:Stokes}, we readily obtain
\begin{align}\label{form-u}
\int_{\Omega} \langle \mathbf{A}(x, \nabla \mathbf{u}), \nabla \varphi \rangle dx = \int_{\Omega} \langle \mathbf{B}(x,\mathbf{f}), \nabla \varphi \rangle dx,
\end{align}
for all $\varphi \in  W^{1,p}_{0,\mathrm{div}}(\Omega;\mathbb{R}^d)$ and further,
\begin{align}\label{form-p}
\int_{\Omega} \langle \mathbf{A}(x, \nabla \mathbf{u}), \nabla \varphi \rangle dx - \int_{\Omega} \mathrm{\pi}\mathrm{div}(\varphi) dx = \int_{\Omega} \langle \mathbf{B}(x,\mathbf{f}), \nabla \varphi \rangle dx,
\end{align}
for all $\varphi \in W^{1,p}_0(\Omega;\mathbb{R}^d)$. In a similar fashion, the pair $(\mathbf{v},\mathrm{\pi}_{\mathbf{v}}) \in W^{1,p}_{\mathbf{u}-\mathbf{g},\mathrm{div}}(\mathfrak{B}_4;\mathbb{R}^d) \times L^{p'}_{\mathrm{int}}(\mathfrak{B}_4)$ is a unique solution to homogeneous system~\eqref{eq:hom-1}, we also have
\begin{align}\label{form-v}
\int_{\mathfrak{B}_4} \langle \mathbf{A}(x, \nabla \mathbf{v}), \nabla \varphi \rangle dx = 0,
\end{align}
for every $\varphi \in  W^{1,p}_{0,\mathrm{div}}(\mathfrak{B}_4;\mathbb{R}^d)$ and
\begin{align}\label{form-pv}
\int_{\mathfrak{B}_4} \langle \mathbf{A}(x, \nabla \mathbf{v}), \nabla \varphi \rangle dx - \int_{\mathfrak{B}_4} \mathrm{\pi}_{\mathbf{v}}\mathrm{div}(\varphi) dx = 0,
\end{align}
We test the weak formulations~\eqref{form-u} and~\eqref{form-v} with function $\varphi = \mathbf{u} -\mathbf{g} - \overline{\mathbf{v}} \in W^{1,p}_{0,\mathrm{div}}(\mathfrak{B}_4;\mathbb{R}^d)$, where $\overline{\mathbf{v}}$ is defined in $\Omega$ according to the form: $\overline{\mathbf{v}} = \mathbf{v}$ in $\mathfrak{B}_4$ and $\overline{\mathbf{v}} = \mathbf{u}-\mathbf{g}$ in $\Omega \setminus \mathfrak{B}_4$, it then yields
\begin{align}\label{form-u-v}
\fint_{\mathfrak{B}_4} \langle \mathbf{A}(x, \nabla \mathbf{u}) - \mathbf{A}(x, \nabla \mathbf{v}),  \nabla (\mathbf{u}-\mathbf{g}-\mathbf{v}) \rangle dx = \fint_{\mathfrak{B}_4} \langle \mathbf{B}(x,\mathbf{f}), \nabla (\mathbf{u}-\mathbf{g}-\mathbf{v})\rangle dx.
\end{align}
Exploiting the structure assumptions~\eqref{cond:A1}-\eqref{cond:A2} and~\eqref{cond:Bf} on both the operators $\mathbf{A}$ and $\mathbf{B}$, from~\eqref{form-u-v}, it enables us to deduce that
\begin{align}\label{est-1a}
\fint_{\mathfrak{B}_4} \Phi(\mathbf{u},\mathbf{v}) dx & \le \Upsilon^2 \fint_{\mathfrak{B}_4} \left[(\sigma^2 + |\nabla \mathbf{u}|^2)^{\frac{p-1}{2}} + (\sigma^2 + |\nabla \mathbf{v}|^2)^{\frac{p-1}{2}}\right] |\nabla \mathbf{g}| dx  \notag \\
& \hspace{1cm} + \Upsilon^2  \fint_{\mathfrak{B}_4} \left(\sigma^2 + |\mathbf{f}|^2\right)^{\frac{p-1}{2}} \left(|\nabla \mathbf{u}-\nabla \mathbf{v}|  + |\nabla \mathbf{g}|\right) dx,
\end{align}
where for the sake of shortness, we introduce the further following function $\Phi: W^{1,p}(\Omega) \to \mathbb{R}^+$ defined as
\begin{align}\label{def:Phi}
\Phi(\mathbf{u},\mathbf{v}) := \left(\sigma^2 + |\nabla \mathbf{u}|^2 + |\nabla \mathbf{v}|^2\right)^{\frac{p-2}{2}} |\nabla \mathbf{u}- \nabla \mathbf{v}|^2.
\end{align}
Applying Young's inequality and reabsorb into the right-hand side of~\eqref{est-1a}, for every $\epsilon_1,\epsilon_2>0$ to arrive at
\begin{align}\label{est-2a}
\fint_{\mathfrak{B}_4} \Phi(\mathbf{u},\mathbf{v}) dx & \le \epsilon_1 \fint_{\mathfrak{B}_4} (\sigma^2 + |\nabla \mathbf{u}|^2)^{\frac{p}{2}} dx + \epsilon_2 \fint_{\mathfrak{B}_4} |\nabla \mathbf{u}-\nabla \mathbf{v}|^p dx \notag \\
& \hspace{1cm} + C \left(\epsilon_1^{-(p-1)} + \epsilon_2^{-\frac{1}{p-1}}\right)  \fint_{\mathfrak{B}_4} \left(\sigma^p + |\mathbf{f}|^p + |\nabla \mathbf{g}|^p\right) dx,
\end{align}
where $C$ is a constant depending only on $\Upsilon,p$. Again, carefully exploiting Young’s inequality for every $\epsilon_3>0$,  it finds $C = C(p)>0$ such that 
\begin{align}\label{est-3a}
|\nabla \mathbf{u}- \nabla \mathbf{v}|^p \le  \chi_{\{p<2\}} \epsilon_3 \left(\sigma^2 + |\nabla \mathbf{u}|^2\right)^{\frac{p}{2}} + C \epsilon_3^{-\tilde{p}} \Phi(\mathbf{u},\mathbf{v}), 
\end{align}
where $\chi_{\{p<2\}}$ and $\tilde{p}$ are defined as 
\begin{align*}
 \chi_{\{p<2\}} = \begin{cases} 1, \ & \mbox{ if } p<2,\\ 0, \ & \mbox { if } p \ge 2,\end{cases} \  \mbox{ and } \ \tilde{p} = \max\left\{0; \frac{2}{p}-1\right\}.
\end{align*}
Combining~\eqref{est-3a} with the comparison estimate from~\eqref{est-2a}, it leads us to obtain
\begin{align}\label{est-4a}
\fint_{\mathfrak{B}_4} |\nabla \mathbf{u} - \nabla \mathbf{v}|^p dx & \le \left( \chi_{\{p<2\}} \epsilon_3 + C \epsilon_3^{-\tilde{p}} \epsilon_1 \right) \fint_{\mathfrak{B}_4} \left(\sigma^2 + |\nabla \mathbf{u}|^2\right)^{\frac{p}{2}} dx \notag \\
& \quad + C \epsilon_3^{-\tilde{p}} \epsilon_2 \fint_{\mathfrak{B}_4} |\nabla \mathbf{u}-\nabla \mathbf{v}|^p dx \notag \\
& \hspace{1cm} + C \epsilon_3^{-\tilde{p}} \left(\epsilon_1^{-(p-1)} + \epsilon_2^{-\frac{1}{p-1}}\right)  \fint_{\mathfrak{B}_4} \left(\sigma^p + |\mathbf{f}|^p + |\nabla \mathbf{g}|^p\right) dx.
\end{align}
Having arrived at this stage, for every $\varepsilon >0$, let us choose $\epsilon_1, \epsilon_2, \epsilon_3>0$ depending on $\varepsilon$  such that
\begin{align*}
\chi_{\{p<2\}} \epsilon_3 + C \epsilon_3^{-\tilde{p}} \epsilon_1 = \frac{\varepsilon}{2}, \quad C \epsilon_3^{-\tilde{p}} \epsilon_2 = \frac{1}{2},
\end{align*}
and then, re-absorb in~\eqref{est-4a}, it yields the following estimate
\begin{align}\label{est-5a}
\fint_{\mathfrak{B}_4} |\nabla \mathbf{u} - \nabla \mathbf{v}|^p dx & \le \varepsilon \fint_{\mathfrak{B}_4} \left(\sigma^2 + |\nabla \mathbf{u}|^2\right)^{\frac{p}{2}} dx \notag \\
& \hspace{1cm} + C \varepsilon^{-\overline{p}} \fint_{\mathfrak{B}_4} \left(\sigma^p + |\mathbf{f}|^p + |\nabla \mathbf{g}|^p\right) dx.
\end{align}
Here, for ease of notation, we denote where $\overline{p}$ by
\begin{align}\label{over-p}
\overline{p} := \max\left\{\frac{p \tilde{p}}{p-1}; \, p \tilde{p}+p-1\right\} = \begin{cases} p-1, & \quad \mbox{if } p \ge 2, \\ 1, & \quad \mbox{if } \frac{3}{2} < p < 2, \\ \frac{2-p}{p-1}, & \quad \mbox{if } 1<p \le \frac{3}{2}. \end{cases}
\end{align}
Next, we will infer a comparison between two associated pressures. From~\eqref{form-p} and~\eqref{form-pv}, it leads us to
\begin{align} 
& \int_{\mathfrak{B}_4} \langle \mathbf{A}(x, \nabla \mathbf{u})-\mathbf{A}(x, \nabla\mathbf{v}), \nabla \varphi \rangle dx - \int_{\mathfrak{B}_4} (\mathrm{\pi}-\mathrm{\pi}_{\mathbf{v}})\mathrm{div}(\varphi) dx  = \int_{\mathfrak{B}_4} \langle \mathbf{B}(x, \mathbf{f}), \nabla \varphi \rangle dx, \notag
\end{align}
for all $\varphi \in W^{1,p}_0(\mathfrak{B}_4;\mathbb{R}^d)$. Finally, the strategy of proof is now to apply~\eqref{est:global-2} in Lemma~\ref{lem:global} to deduce
\begin{align}\notag 
& \fint_{\mathfrak{B}_4}|\mathrm{\pi} - \mathrm{\pi}_{\mathbf{v}}|^{p'} dx \le C \fint_{\mathfrak{B}_4}\left(\sigma^p + |\mathbf{f}|^p + |\nabla \mathbf{u} - \nabla \mathbf{v}|^p\right) dx, 
\end{align}
and the desired result follows by~\eqref{est-5a}. This finishes the proof
\end{proof}

The second step consists in considering homogeneous constant coefficients problem and establishing the comparison estimate between weak solution pairs. A bit more precise, we shall focus on the new problem where the matrix-valued function $\mathbf{A}$ just keeps integral average over $B_{2R}(x_0)$ entries. This is the content of the following comparison lemma.
\begin{lemma}\label{lem:comp-2}
Under the assumptions of Lemma~\ref{lem:comp-1}, we further denote $\mathfrak{B}_1 := B_{R}(x_0)$ and $\mathfrak{B}_2 := B_{2R}(x_0)$. Assume that $(\mathbf{w},\mathrm{\pi}_{\mathbf{w}}) \in W^{1,p}_{\mathbf{v},\mathrm{div}}(\mathfrak{B}_2;\mathbb{R}^d) \times L^{p'}_{\mathrm{int}}(\mathfrak{B}_2)$ is the unique solution to the following homogeneous constant coefficients problem
\begin{align}\label{eq:hom-2}
\begin{cases}
-\mathrm{div}(\overline{\mathbf{A}}_{\mathfrak{B}_2}(\nabla \mathbf{w})) + \nabla \mathrm{\pi}_{\mathbf{w}} &= \ 0 \hspace{1.25cm} \ \, \text{in}\ \mathfrak{B}_2, \\
 \hspace{1.4cm} \mathrm{div} (\mathbf{w}) &= \ 0  \hspace{1.3cm} \ \text{in}\ \mathfrak{B}_2, \\
 \hspace{1.8cm} \mathbf{w} & = \ \mathbf{v} \hspace{1.4cm} \text{on}\ \partial \mathfrak{B}_2.
 \end{cases}
\end{align}
Then, there exist $\kappa=\kappa(d,\Upsilon,p)>0$ and $C = C(d,\Upsilon,p)>0$ such that whenever $[\mathbf{A}]_{R_0} \le \delta$ for some $\delta \in (0,1)$, the estimate
\begin{align}\label{est-comp-2}
 \fint_{\mathfrak{B}_2}|\nabla \mathbf{u} - \nabla \mathbf{w}|^p + |\mathrm{\pi} - \mathrm{\pi}_{\mathbf{w}}|^{p'} dx  & \le C \delta^{\kappa} \fint_{\mathfrak{B}_4} \left(\sigma^2 + |\nabla \mathbf{u}|^2\right)^{\frac{p}{2}} dx \notag \\
& \hspace{1cm} + C \delta^{-\overline{p}\kappa} \fint_{\mathfrak{B}_4} \left(\sigma^p + |\mathbf{f}|^p + |\nabla \mathbf{g}|^p\right) dx
\end{align}
holds true, where $\overline{p}$ has been defined in~\eqref{over-p}. Moreover, there holds
\begin{align}\label{est-comp-3}
 \|\nabla \mathbf{w}\|_{L^{\infty}(\mathfrak{B}_1;\mathbb{R}^{d^2})}^p + \|\mathrm{\pi}_{\mathbf{w}}\|_{L^{\infty}(\mathfrak{B}_1)}^{p'} &\le C\fint_{\mathfrak{B}_4} \left(\sigma^2 + |\nabla \mathbf{u}|^2\right)^{\frac{p}{2}} dx \notag \\
& \hspace{1cm}  + C \fint_{\mathfrak{B}_4} \left(\sigma^p + |\mathbf{f}|^p + |\nabla \mathbf{g}|^p\right) dx,
\end{align}
where the constant $C$ depends upon $(d,\Upsilon,p)$.
\end{lemma}
\begin{proof}
First, we go once back to the homogeneous problem~\eqref{eq:hom-1} with a unique solution $(\mathbf{v},\mathrm{\pi}_{\mathbf{v}})$ shown in Lemma~\ref{lem:comp-1}. From problems~\eqref{eq:hom-1} and~\eqref{eq:hom-2}, it follows that
\begin{align}\label{eq:hom-3}
\begin{cases}
-\mathrm{div}(\overline{\mathbf{A}}_{\mathfrak{B}_2}(\nabla \mathbf{v}) -\overline{\mathbf{A}}_{\mathfrak{B}_2}(\nabla \mathbf{w})) + \nabla (\mathrm{\pi}_{\mathbf{v}}-\mathrm{\pi}_{\mathbf{w}})  &= \ \mathrm{div}(\mathbf{A}(x,\nabla \mathbf{v}) - \overline{\mathbf{A}}_{\mathfrak{B}_2}(\nabla \mathbf{v})) \quad  \text{in}\ \mathfrak{B}_2, \\
 \hspace{2cm} \mathrm{div} (\mathbf{v}-\mathbf{w}) &= \ 0  \hspace{4.55cm} \ \text{in}\ \mathfrak{B}_2, \\
 \hspace{2.4cm} \mathbf{v}-\mathbf{w} & = \ 0 \hspace{4.65cm} \text{on}\ \partial \mathfrak{B}_2.
 \end{cases}
\end{align}
Testing the weak form of~\eqref{eq:hom-3} with $\varphi = \mathbf{v} - \mathbf{w} \in W^{1,p}_{0,\mathrm{div}}(\mathfrak{B}_2;\mathbb{R}^d)$, we achieve
\begin{align}\notag 
\fint_{\mathfrak{B}_2} \langle \overline{\mathbf{A}}_{\mathfrak{B}_2}(\nabla \mathbf{v}) -\overline{\mathbf{A}}_{\mathfrak{B}_2}(\nabla \mathbf{w}),  \nabla (\mathbf{v}-\mathbf{w}) \rangle dx = \fint_{\mathfrak{B}_2} \langle \mathbf{A}(x,\nabla \mathbf{v}) - \overline{\mathbf{A}}_{\mathfrak{B}_2}(\nabla \mathbf{v}), \nabla (\mathbf{v}-\mathbf{w})\rangle dx.
\end{align}
Thanks to assumption~\eqref{cond:A2}, it results that
\begin{align}\label{v-w-1}
& \fint_{\mathfrak{B}_2} \Phi(\mathbf{v},\mathbf{w}) dx \le \Upsilon \fint_{\mathfrak{B}_2} \mathcal{O}(x) \left(\sigma^2 + |\nabla \mathbf{v}|^2\right)^{\frac{p-1}{2}} |\nabla (\mathbf{v}-\mathbf{w})|  dx, 
\end{align}
where $\Phi$ is just given in~\eqref{def:Phi}, the new function $\mathcal{O}$ defined by
\begin{align*}
\mathcal{O}(x) = \sup_{\nu \in \mathbb{R}^{d^2}\setminus\{0\}} \frac{|\mathbf{A}(x,\nu) - \overline{\mathbf{A}}_{\mathfrak{B}_2}(\nu)|}{(\sigma^2 +|\nu|^2)^{\frac{p-1}{2}}}.
\end{align*}
Using H\"older's inequality for the right-hand side of~\eqref{v-w-1}, we therefore deduce that for every $\epsilon_1>0$ 
\begin{align}\label{v-w-1.1}
\fint_{\mathfrak{B}_2} \Phi(\mathbf{v},\mathbf{w}) dx & \le \epsilon_1 \fint_{\mathfrak{B}_2} |\nabla \mathbf{v}- \nabla \mathbf{w}|^p dx +  C \epsilon_1^{-\frac{1}{p-1}} \fint_{\mathfrak{B}_2} [\mathcal{O}(x)]^{\frac{p}{p-1}} \left(\sigma^2 + |\nabla \mathbf{v}|^2\right)^{\frac{p}{2}} dx.
\end{align}
The standard higher integrability of $\mathbf{v}$ gives us the existence of $\theta=\theta(d,\Upsilon,p)>p$ such that
\begin{align}\label{RH-v}
\fint_{\mathfrak{B}_2} (\sigma^2 + |\nabla \mathbf{v}|^2)^{\frac{\theta}{2}}dx \le C\left(\fint_{\mathfrak{B}_4} (\sigma^2 + |\nabla \mathbf{v}|^2)^{\frac{p}{2}}dx\right)^{\frac{\theta}{p}}.
\end{align}
Combining~\eqref{RH-v} with~\eqref{v-w-1.1} and using again H\"older's inequality, it gives
\begin{align}
 \fint_{\mathfrak{B}_2} \Phi(\mathbf{v},\mathbf{w}) dx & \le \epsilon_1 \fint_{\mathfrak{B}_2} |\nabla \mathbf{v}- \nabla \mathbf{w}|^p dx \notag \\
& \qquad +  C \epsilon_1^{-\frac{1}{p-1}} \left(\fint_{\mathfrak{B}_2} [\mathcal{O}(x)]^{\frac{p}{p-1} \frac{\theta}{\theta-p}} dx\right)^{1-\frac{p}{\theta}} \left(\fint_{\mathfrak{B}_2} \left(\sigma^2 + |\nabla \mathbf{v}|^2\right)^{\frac{\theta}{2}} dx\right)^{\frac{p}{\theta}} \notag \\
& \le \epsilon_1 \fint_{\mathfrak{B}_2} |\nabla \mathbf{v}- \nabla \mathbf{w}|^p dx \notag \\
& \qquad +  C \epsilon_1^{-\frac{1}{p-1}} \left(\fint_{\mathfrak{B}_2} \mathcal{O}(x) dx\right)^{1-\frac{p}{\theta}} \left(\fint_{\mathfrak{B}_4} \left(\sigma^2 + |\nabla \mathbf{v}|^2\right)^{\frac{p}{2}} dx\right).\notag
\end{align}
Here, it should be noted that under the structure condition~\eqref{cond:A1}, $\mathcal{O}(x) \le 2\Upsilon$ for every $x \in \Omega$. Further, since $[\mathbf{A}]_{R_0} \le \delta$, we infer that
\begin{align}\label{v-w-3}
 \fint_{\mathfrak{B}_2} \Phi(\mathbf{v},\mathbf{w}) dx & \le \epsilon_1 \fint_{\mathfrak{B}_2} |\nabla \mathbf{v}- \nabla \mathbf{w}|^p dx  +  C \epsilon_1^{-\frac{1}{p-1}} {\delta}^{1-\frac{p}{\theta}} \fint_{\mathfrak{B}_4} \left(\sigma^2 + |\nabla \mathbf{v}|^2\right)^{\frac{p}{2}} dx.
\end{align}
By the same argument as in~\eqref{est-3a}, for every $\epsilon_2>0$, there holds
\begin{align}\notag
|\nabla \mathbf{v}- \nabla \mathbf{w}|^p \le  \chi_{\{p<2\}} \epsilon_2 \left(\sigma^2 + |\nabla \mathbf{v}|^2\right)^{\frac{p}{2}} + C \epsilon_2^{-\tilde{p}} \Phi(\mathbf{v},\mathbf{w}), 
\end{align}
where $\tilde{p}$ given in~\eqref{est-3a}. As a result, it follows from~\eqref{v-w-3} that
\begin{align}\label{v-w-4}
\fint_{\mathfrak{B}_2} |\nabla \mathbf{v}- \nabla \mathbf{w}|^p dx  & \le C\left( \chi_{\{p<2\}} \epsilon_2 + \epsilon_2^{-\tilde{p}} \epsilon_1^{-\frac{1}{p-1}} {\delta}^{1-\frac{p}{\theta}}\right) \fint_{\mathfrak{B}_4} \left(\sigma^2 + |\nabla \mathbf{v}|^2\right)^{\frac{p}{2}} dx \notag \\
& \qquad + C \epsilon_2^{-\tilde{p}} \epsilon_1 \fint_{\mathfrak{B}_2} |\nabla \mathbf{v}- \nabla \mathbf{w}|^p dx.
\end{align}
Now, let us fix $\epsilon_1$ and $\epsilon_2$ such that 
$$C \epsilon_2^{-\tilde{p}} \epsilon_1 \le 1/2, \mbox{ and } \epsilon_2 = \epsilon_2^{-\frac{\tilde{p} p}{p-1}} {\delta}^{1-\frac{p}{\theta}} \Leftrightarrow \epsilon_2 = {\delta}^{\left(1-\frac{p}{\theta}\right)\frac{p-1}{\tilde{p} p+p-1}},$$
and then, it can be deduced from~\eqref{v-w-4} that
\begin{align}\label{v-w-5}
\fint_{\mathfrak{B}_2} |\nabla \mathbf{v}- \nabla \mathbf{w}|^p dx  & \le C {\delta}^{\frac{(\theta-p)(p-1)}{\theta(\tilde{p} p+p-1)}} \fint_{\mathfrak{B}_4} \left(\sigma^2 + |\nabla \mathbf{v}|^2\right)^{\frac{p}{2}} dx.
\end{align}
In order to simplify the notation, we shall denote
\begin{align}\label{def:kappa}
\kappa := \frac{(\theta-p)(p-1)}{\theta(\tilde{p} p+p-1)}.
\end{align}
By the same way as the proof of~\eqref{est:global-2} in Lemma~\ref{lem:global}, combining~\eqref{v-w-5} and~\eqref{def:kappa} together leads us to
\begin{align}\notag
& \fint_{\mathfrak{B}_2}|\nabla \mathbf{v} - \nabla \mathbf{w}|^p + |\mathrm{\pi}_{\mathbf{v}} - \mathrm{\pi}_{\mathbf{w}}|^{p'} dx \notag \le C {\delta}^{\kappa} \fint_{\mathfrak{B}_4} \left(\sigma^2 + |\nabla \mathbf{v}|^2\right)^{\frac{p}{2}} dx \notag \\
& \hspace{4cm} \le C {\delta}^{\kappa} \fint_{\mathfrak{B}_4} \left(\sigma^2 + |\nabla \mathbf{u}|^2\right)^{\frac{p}{2}} dx + C {\delta}^{\kappa} \fint_{\mathfrak{B}_4} |\nabla \mathbf{u} - \nabla \mathbf{v}|^p dx. \label{v-w-6}
\end{align}
On the other hand, applying Lemma~\ref{lem:comp-1} for $\varepsilon = {\delta}^{\kappa}$ yields
\begin{align}\label{v-w-7}
 \fint_{\mathfrak{B}_4}|\nabla \mathbf{u} - \nabla \mathbf{v}|^p + |\mathrm{\pi} - \mathrm{\pi}_{\mathbf{v}}|^{p'} dx  & \le {\delta}^{\kappa} \fint_{\mathfrak{B}_4} \left(\sigma^2 + |\nabla \mathbf{u}|^2\right)^{\frac{p}{2}} dx \notag \\
& \hspace{1cm} + C {\delta}^{-\overline{p}\kappa} \fint_{\mathfrak{B}_4} \left(\sigma^p + |\mathbf{f}|^p + |\nabla \mathbf{g}|^p\right) dx,
\end{align}
where $\overline{p}$ given in~\eqref{over-p}. Having arrived at this stage, combining~\eqref{v-w-6} and~\eqref{v-w-7} leads to
\begin{align}\label{v-w-8}
\fint_{\mathfrak{B}_2}|\nabla \mathbf{u} - \nabla \mathbf{w}|^p  + |\mathrm{\pi} - \mathrm{\pi}_{\mathbf{w}}|^{p'} dx & \le C\fint_{\mathfrak{B}_2}|\nabla \mathbf{u} - \nabla \mathbf{v}|^p + |\mathrm{\pi} - \mathrm{\pi}_{\mathbf{v}}|^{p'} dx \notag \\
& \hspace{1cm} + C\fint_{\mathfrak{B}_2}|\nabla \mathbf{v} - \nabla \mathbf{w}|^p + |\mathrm{\pi}_{\mathbf{v}} - \mathrm{\pi}_{\mathbf{w}}|^{p'} dx, \notag \\
& \le C {\delta}^{\kappa} \fint_{\mathfrak{B}_4} \left(\sigma^2 + |\nabla \mathbf{u}|^2\right)^{\frac{p}{2}} dx \notag \\
& \hspace{1cm} + C \varepsilon^{-\overline{p} \kappa} \fint_{\mathfrak{B}_4} \left(\sigma^p + |\mathbf{f}|^p + |\nabla \mathbf{g}|^p\right) dx,
\end{align}
which concludes~\eqref{est-comp-2}. Finally, in order to prove~\eqref{est-comp-3}, we take the use of a standard regularity for solution of problem~\eqref{eq:hom-3} as in~\cite{GM82} to ensure that
\begin{align}\notag
& \|\nabla \mathbf{w}\|_{L^{\infty}(\mathfrak{B}_1;\mathbb{R}^{d^2})}^p + \|\mathrm{\pi}_{\mathbf{w}}\|_{L^{\infty}(\mathfrak{B}_1)}^{p'} \le C\fint_{\mathfrak{B}_2} \left(\sigma^2 + |\nabla \mathbf{w}|^2\right)^\frac{p}{2} dx.
\end{align}
The proof is now complete by applying~\eqref{v-w-8}.  
\end{proof}

\subsection{Up-to-the-boundary comparison scheme}
\label{sec:boundary}

The next step of comparison scheme specifies the up-to-boundary extension of the interior results in Lemma~\ref{lem:comp-1} and~\ref{lem:comp-2}.  In general, to prove comparison estimates near the boundary of $\Omega$, one could go through the same process as interior proofs. However, the main difficulty is to estimate the $L^\infty$-gradient bounds near the boundary, especially when $\partial\Omega$ is not sufficiently smooth. Therefore, the structural assumptions~\eqref{cond:Reif} (Reifenberg flatness condition) and~\eqref{BMO} (the small BMO condition) are further imposed to our problem, i.e. 
 $(\Omega,\mathbf{A}) \in \mathcal{H}(\delta,R_0)$, for some $\delta \in (0,1/81)$ and $R_0>0$. 
 

The proof of boundary version is non-trivial. By approximation scheme, we are able to approximate a portion of the Reifenberg-flat boundary by an $(d-1)$-dimensional hyper-plane. Before going into details let us briefly describe the idea. At a fixed point $x_0 \in \partial\Omega$ and $R \in (0,R_0/9)$, Reifenberg flat assumption allows us to find a new coordinate system with its origin $O$ belonging to $\Omega$ and a ball $\tilde{B}_{\varrho}(O)$ with radius $\varrho>0$ such that its upper-half $\tilde{B}_{\varrho}^+(O)$ satisfying
\begin{align*}
\tilde{B}_{\varrho}^+(O) \subset \Omega, \quad \text{and} \quad x_0 \not\in \tilde{B}_{\varrho}^+(O) \cap \Omega.
\end{align*}
On this intersection $\tilde{B}_{\varrho}^+(O) \cap \Omega$, under these structural assumptions, we are able to  estimate the $L^\infty$-norm for solution of homogeneous boundary value problem. 

To be precise, this will be done in a multi-step comparison argument. First, for each $k \in \mathbb{N}$, we denote by $\Omega_{k} := B_{kR}(x_0) \cap \Omega$ the ``surface ball'' of $\Omega$ that associated with $B_{kR}(x_0)$. Similar to the interior case, we construct the comparison between weak solution $(\mathbf{u},\pi_{\mathbf{u}})$ and the unique solution $(\mathbf{v},\mathrm{\pi}_{\mathbf{v}}) \in W^{1,p}_{\mathbf{u}-\mathbf{g},\mathrm{div}}(\Omega_{10};\mathbb{R}^d) \times L^{p'}_{\mathrm{int}}(\Omega_{10})$ of the corresponding homogeneous boundary-value problem 
\begin{align*}
\begin{cases}
-\mathrm{div}(\mathbf{A}(x,\nabla \mathbf{v})) + \nabla \mathrm{\pi}_{\mathbf{v}} &= \ 0 \hspace{1.25cm} \ \, \text{in}\ \Omega_{10}, \\
 \hspace{1.4cm} \mathrm{div} (\mathbf{v}) &= \ 0  \hspace{1.3cm} \ \text{in}\ \Omega_{10}, \\
 \hspace{1.8cm} \mathbf{v} & = \ \mathbf{u}-\mathbf{g} \hspace{0.7cm} \text{on}\ \partial \Omega_{10}.
 \end{cases}
\end{align*}
Next, due to the dilation argument, Reifenberg flatness condition and the small BMO semi-norm, we compare $(\mathbf{v},\pi_{\mathbf{v}})$ to the (unique) solution $(\tilde{\mathbf{v}},\mathrm{\pi}_{\tilde{\mathbf{v}}}) \in W^{1,p}_{\mathbf{v},\mathrm{div}}(\tilde{B}_\varrho(O) \cap \Omega;\mathbb{R}^d) \times L^{p'}_{\mathrm{int}}(\tilde{B}_\varrho(O) \cap \Omega)$ of the following Dirichlet problem in a smaller domain with the flat boundary
\begin{align*}
\begin{cases}
-\mathrm{div}(\overline{\mathbf{A}}_{\tilde{B}_{\varrho}}(\nabla \tilde{\mathbf{v}})) + \nabla \mathrm{\pi}_{\tilde{\mathbf{v}}} &= \ 0 \hspace{1cm} \ \, \text{in}\ \tilde{B}_{\varrho} \cap \Omega, \\
 \hspace{1.3cm} \mathrm{div} (\tilde{\mathbf{v}}) &= \ 0  \hspace{1.05cm} \ \text{in}\ \tilde{B}_{\varrho} \cap \Omega, \\
 \hspace{1.5cm} \tilde{\mathbf{v}} & = \ \mathbf{v} \hspace{1.1cm} \text{on}\ \partial (\tilde{B}_{\varrho} \cap \Omega).
 \end{cases}
\end{align*}
We remark here that in the new coordinate system, when no confusion arises, we shall omit the dependence on the center $O$ as follows: $\tilde{B}_{\varrho}:= \tilde{B}_{\varrho}(O)$ and $\tilde{B}_{\varrho}^+:= \tilde{B}_{\varrho} \cap \{\tilde{x}_d > 0\}$. Having arrived at this stage, we observe that in order to establish an $L^\infty$-gradient bound for solutions near the boundary $\partial\Omega$ under assumption~\ref{ass:Reif}, we restrict ourselves to the model case in a half-ball $\tilde{B}_\varrho^+ \subset \Omega$ (we remark that the origin $O$ is in a neighborhood of $x_0 \in \partial\Omega$). Thus, we turn our attention to weak solution pair $(\mathbf{w}, \mathrm{\pi}_{\mathbf{w}})$ of the limiting problem in a smaller domain with the flat boundary
\begin{align*}
\begin{cases}
-\mathrm{div}(\overline{\mathbf{A}}_{\tilde{B}_{\varrho}^+}(\nabla \mathbf{w})) + \nabla \mathrm{\pi}_{\mathbf{w}} &= \ 0 \hspace{.1cm} \ \, \text{in}\  \tilde{B}_{\varrho}^+, \\
 \hspace{1.3cm} \mathrm{div} (\mathbf{w}) &= \ 0  \hspace{0.15cm} \ \text{in}\ \tilde{B}_{\varrho}^+, \\
 \hspace{1.7cm} \mathbf{w} & = \ 0 \hspace{0.3cm} \text{on}\  \tilde{B}_{\varrho}^+ \cap \{\tilde{x}_d=0\}.
 \end{cases}
\end{align*}
On the flat boundary of $\Omega$, let $(\tilde{\mathbf{w}},\pi_{\tilde{\mathbf{w}}})$ be zero extension of $(\mathbf{w},\pi_{\mathbf{w}})$ from $\tilde{B}_\varrho^+$ to $\tilde{B}_\varrho \cap \Omega$. Now, it is of interest to derive the comparison result between two pairs $(\tilde{\mathbf{v}},\mathrm{\pi}_{\tilde{\mathbf{v}}})$ and $(\tilde{\mathbf{w}},\mathrm{\pi}_{\tilde{\mathbf{w}}})$. As a last step of boundary comparison procedures, we will additionally establish $L^\infty$-bound for the gradient of $(\tilde{\mathbf{w}},\mathrm{\pi}_{\tilde{\mathbf{w}}})$. 

\begin{figure}[H]
\centering
\includegraphics[width=\linewidth]{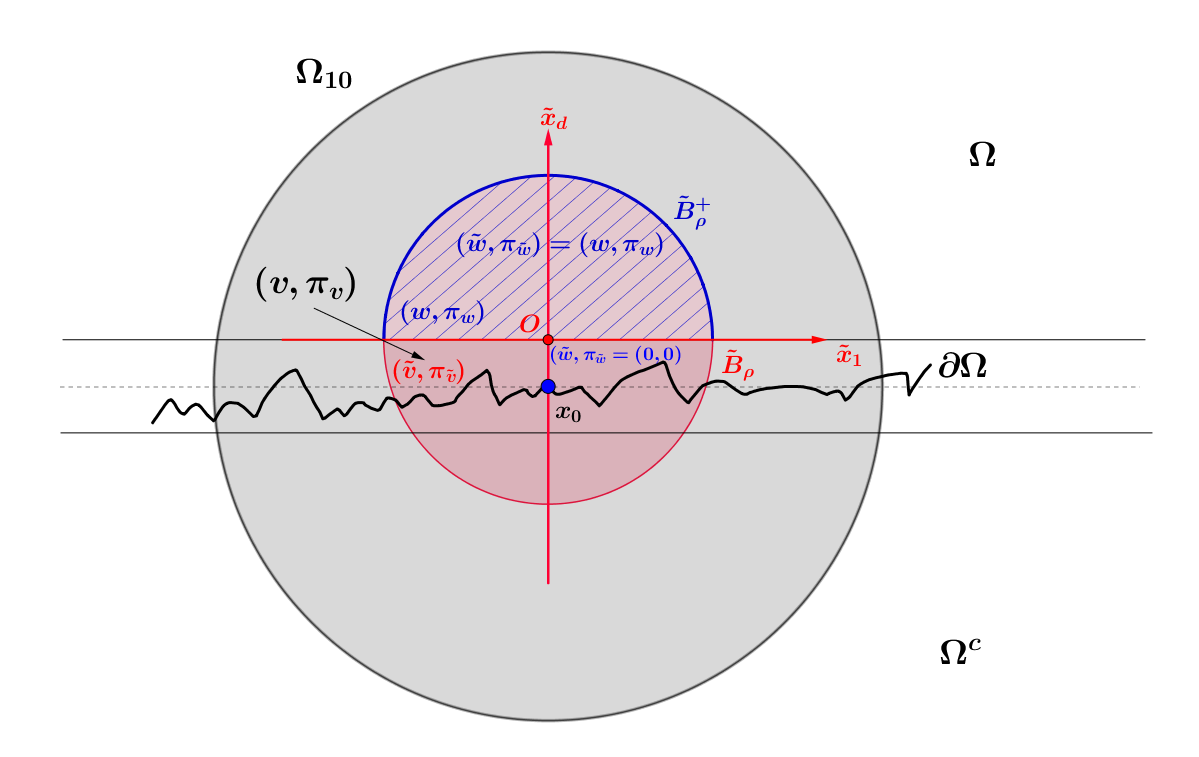}
\caption{Illustration of the domain near the boundary of an $(\delta,R_0)$-Reifenberg flat domain $\Omega$, for which the solution pair $(\tilde{\mathbf{w}},\pi_{\tilde{\mathbf{w}}})$ is defined.}
\label{fig:sols}
\end{figure}

In Figure~\ref{fig:sols}, we illustrate how the comparison scheme is performed on the boundary of $\Omega$ under Reifenberg flatness condition. Further, Figure~\ref{fig:sols} also describes how each solution pair from the discussion above is defined in the particular domain. We record our observations in a lemma, the boundary version of interior comparison estimates in Section~\ref{sec:interior}. As we shall see below, Lemma~\ref{lem:boundary} summarizes and derives comparison results near the boundary of $\Omega$. 

\begin{lemma}\label{lem:boundary}
Let $(\mathbf{u},\mathrm{\pi}) \in W^{1,p}_{\mathbf{g},\mathrm{div}}(\Omega;\mathbb{R}^d) \times L^{p'}_{\mathrm{int}}(\Omega)$ be a weak solution pair to system~\eqref{eq:Stokes} with given data $\mathbf{f} \in L^p(\Omega;\mathbb{R}^{d^2})$ and $\mathbf{g} \in W^{1,p}(\Omega;\mathbb{R}^{d})$. Assume further that $(\Omega,\mathbf{A}) \in \mathcal{H}(\delta,R_0)$, for some $\delta \in (0,1/81)$ and $R_0>0$. Then, for $x_0 \in \partial \Omega$ and $R \in (0,R_0/9)$, one can find $(\tilde{\mathbf{w}},\mathrm{\pi}_{\tilde{\mathbf{w}}}) \in W^{1,p}(\Omega_{1};\mathbb{R}^d) \times L^{p'}(\Omega_{1})$ such that the following estimates hold:
\begin{align}\label{est-boundary-1}
\fint_{\Omega_{1}}|\nabla \mathbf{u} - \nabla \tilde{\mathbf{w}}|^pdx + \fint_{\Omega_{1}}|\mathrm{\pi} - \mathrm{\pi}_{\tilde{\mathbf{w}}}|^{p'} dx & \le C\delta^{\kappa}\fint_{\Omega_{10}} \left(\sigma^2 + |\nabla \mathbf{u}|^2\right)^{\frac{p}{2}}dx \notag \\
& \qquad  + C  \fint_{\Omega_{10}} (\sigma^p + |\mathbf{f}|^p + |\nabla \mathbf{g}|^p)dx,
\end{align}
and 
\begin{align}\label{est-boundary-2}
\|\nabla \tilde{\mathbf{w}}\|_{L^{\infty}(\Omega_{1};\mathbb{R}^{d^2})}^p + \|\mathrm{\pi}_{\tilde{\mathbf{w}}}\|_{L^{\infty}(\Omega_{1})}^{p'} & \le C\fint_{\Omega_{10}} \left(\sigma^2 + |\nabla \mathbf{u}|^2\right)^{\frac{p}{2}}dx \notag \\
& \hspace{1cm}  + C \fint_{\Omega_{10}} (\sigma^p + |\mathbf{f}|^p + |\nabla \mathbf{g}|^p)dx,
\end{align}
where $\kappa$ is defined as in~\eqref{def:kappa}.
\end{lemma}
\begin{proof}
The proof follows along the same path as of Lemma~\ref{lem:comp-1}. Let us first consider $(\mathbf{v},\mathrm{\pi}_{\mathbf{v}}) \in W^{1,p}_{\mathbf{u}-\mathbf{g},\mathrm{div}}(\Omega_{10};\mathbb{R}^d) \times L^{p'}_{\mathrm{int}}(\Omega_{10})$ the unique solution pair of the following Cauchy-Dirichlet problem
\begin{align} \notag 
\begin{cases}
-\mathrm{div}(\mathbf{A}(x,\nabla \mathbf{v})) + \nabla \mathrm{\pi}_{\mathbf{v}} &= \ 0 \hspace{1.25cm} \ \, \text{in}\ \Omega_{10}, \\
 \hspace{1.4cm} \mathrm{div} (\mathbf{v}) &= \ 0  \hspace{1.3cm} \ \text{in}\ \Omega_{10}, \\
 \hspace{1.8cm} \mathbf{v} & = \ \mathbf{u}-\mathbf{g} \hspace{0.7cm} \text{on}\ \partial \Omega_{10}.
 \end{cases}
\end{align}
Applying a similar argument as in the previous estimates, we conclude
\begin{align}\label{est-100}
 \fint_{\Omega_{10}}|\nabla \mathbf{u} - \nabla \mathbf{v}|^p + |\mathrm{\pi} - \mathrm{\pi}_{\mathbf{v}}|^{p'} dx  & \le \varepsilon \fint_{\Omega_{10}} \left(\sigma^2 + |\nabla \mathbf{u}|^2\right)^{\frac{p}{2}} dx \notag \\
& \hspace{1cm} + C \varepsilon^{-\overline{p}} \fint_{\Omega_{10}} \left(\sigma^p + |\mathbf{f}|^p + |\nabla \mathbf{g}|^p\right) dx,
\end{align}
for every $\varepsilon \in (0,1)$. At this stage, we set $\varrho:= 9R(1-\delta)>0$, and due to the assumption~\eqref{cond:Reif}, there exists a new coordinate system $\{\tilde{x}_1, \tilde{x}_2, ..., \tilde{x}_d\}$ such that its origin $O$ belonging to $\Omega$, $x_0 = - \varrho\delta/(1-\delta)\tilde{x}_d$ and
\begin{align}\notag 
\tilde{B}_{\varrho}^+ \subset \tilde{B}_{\varrho} \cap \Omega \subset \tilde{B}_{\varrho} \cap \{\tilde{x}_d > -2 \varrho\delta/(1-\delta)\}.
\end{align}
Since $\varrho= 9R(1-\delta)$ for $\delta \in (0,1/81)$, it is easy to check that
$$ \varrho\delta/(1-\delta) + R < \frac{\varrho}{8}, \mbox{ and } \varrho\delta/(1-\delta) + \varrho < 10R,$$
and this implies the following relations
\begin{align}\label{cond:B2}
\Omega_1= B_{R}(x_0)\cap \Omega \subset \tilde{B}_{\varrho/8}\cap \Omega \subset \tilde{B}_{\varrho/4}\cap \Omega \subset \tilde{B}_{\varrho}\cap \Omega \subset \Omega_{10}= B_{10R}(x_0) \cap \Omega.
\end{align}

Next, let $(\tilde{\mathbf{v}},\mathrm{\pi}_{\tilde{\mathbf{v}}})$ be the unique pair of weak solutions to the following homogeneous boundary-value problem
\begin{align}\notag 
\begin{cases}
-\mathrm{div}(\overline{\mathbf{A}}_{\tilde{B}_{\varrho}}(\nabla \tilde{\mathbf{v}})) + \nabla \mathrm{\pi}_{\tilde{\mathbf{v}}} &= \ 0 \hspace{1cm} \ \, \text{in}\ \tilde{B}_{\varrho} \cap \Omega, \\
 \hspace{1.3cm} \mathrm{div} (\tilde{\mathbf{v}}) &= \ 0  \hspace{1.05cm} \ \text{in}\ \tilde{B}_{\varrho} \cap \Omega, \\
 \hspace{1.5cm} \tilde{\mathbf{v}} & = \ \mathbf{v} \hspace{1.1cm} \text{on}\ \partial (\tilde{B}_{\varrho} \cap \Omega).
 \end{cases}
\end{align}
Note that it is convenient to extend $\mathbf{A}(\cdot,\nu)$ to be zero outside $\Omega$ and this guarantees that $\overline{\mathbf{A}}_{\tilde{B}_{\varrho}}(\nu)$ is still well-defined. A brief argument similar to the one in Lemma~\ref{lem:comp-2} leads us now to the estimate
\begin{align}\label{est_lem:v-w500a}
\fint_{\Omega \cap \tilde{B}_{\varrho/4}} |\nabla \mathbf{v} - \nabla \tilde{\mathbf{v}}|^p + |\mathrm{\pi}_{\mathbf{v}} - \mathrm{\pi}_{\tilde{\mathbf{v}}}|^{p'} dx & \le C \delta^{\kappa} \fint_{\Omega \cap \tilde{B}_{\varrho/2}} \left(\sigma^2 + |\nabla \mathbf{v}|^2\right)^{\frac{p}{2}} dx \notag \\
& \le C \delta^{\kappa} \left( \fint_{\Omega_{10}} \left(\sigma^2 + |\nabla \mathbf{u}|^2\right)^{\frac{p}{2}} dx + \fint_{\Omega_{10}} |\nabla \mathbf{u}-\nabla \mathbf{v}|^pdx\right).
\end{align}
Substituting~\eqref{est-100} into~\eqref{est_lem:v-w500a} with $\varepsilon=1/2$, we eventually obtain
\begin{align}\label{est_lem:v-w500}
\fint_{\Omega \cap \tilde{B}_{\varrho/4}} |\nabla \mathbf{v} - \nabla \tilde{\mathbf{v}}|^p + |\mathrm{\pi}_{\mathbf{v}} - \mathrm{\pi}_{\tilde{\mathbf{v}}}|^{p'} dx & \le C \delta^{\kappa}  \fint_{\Omega_{10}} \left(\sigma^2 + |\nabla \mathbf{u}|^2\right)^{\frac{p}{2}} dx \notag \\
& \qquad + C \delta^{\kappa} \fint_{\Omega_{10}} \left(\sigma^p + |\mathbf{f}|^p + |\nabla \mathbf{g}|^p\right) dx.
\end{align}
At this point, it is not possible to estimate the $L^{\infty}$-bound of $(\nabla \tilde{\mathbf{v}},\mathrm{\pi}_{\tilde{\mathbf{v}}})$ near the boundary of irregular domain $\Omega$. Under Reifenberg flat condition~\eqref{cond:Reif}, in the new coordinate system, it allows us to study  the limiting system with zero Dirichlet boundary data in $\tilde{B}_{\varrho}^+ \subset \Omega$. To be more precise, let us consider $(\mathbf{w}, \mathrm{\pi}_{\mathbf{w}})$ a weak solution to the following homogeneous Dirichlet problem
\begin{align}\label{eq:B4}
\begin{cases}
-\mathrm{div}(\overline{\mathbf{A}}_{\tilde{B}_{\varrho}^+}(\nabla \mathbf{w})) + \nabla \mathrm{\pi}_{\mathbf{w}} &= \ 0 \hspace{.1cm} \ \, \text{in}\  \tilde{B}_{\varrho}^+, \\
 \hspace{1.3cm} \mathrm{div} (\mathbf{w}) &= \ 0  \hspace{0.15cm} \ \text{in}\ \tilde{B}_{\varrho}^+, \\
 \hspace{1.7cm} \mathbf{w} & = \ 0 \hspace{0.3cm} \text{on}\  \tilde{B}_{\varrho} \cap \{\tilde{x}_d=0\}.
 \end{cases}
\end{align}
We now denote by $(\tilde{\mathbf{w}},\mathrm{\pi}_{\tilde{\mathbf{w}}})$ the zero extension of $(\mathbf{w}, \mathrm{\pi}_{\mathbf{w}})$ from $\tilde{B}_{\varrho}^+$ to $\tilde{B}_{\varrho} \cap \Omega$. An interesting point here is that the $L^{\infty}$-norm estimate of $(\nabla \tilde{\mathbf{w}},\mathrm{\pi}_{\tilde{\mathbf{w}}})$ for the limiting problem~\eqref{eq:B4} is well-known, see for instance~\cite{GM82}. In particular, we will obtain
\begin{align}
\|\nabla \tilde{\mathbf{w}}\|^p_{L^{\infty}(\tilde{B}_{\varrho/8}\cap \Omega;\mathbb{R}^{d^2})} + \|\mathrm{\pi}_{\tilde{\mathbf{w}}}\|^{p'}_{L^{\infty}(\tilde{B}_{\varrho/8}\cap \Omega)} &\le C \fint_{\tilde{B}_{\varrho/4}\cap \Omega} \left(\sigma^2 + |\nabla \mathbf{v}|^2\right)^{\frac{p}{2}} dx \notag \\
& \le C \fint_{\tilde{B}_{\varrho/4}\cap \Omega} \left(\sigma^2+|\nabla \mathbf{u}|^2\right)^{\frac{p}{2}} + |\nabla \mathbf{u} - \nabla \mathbf{v}|^p dx, \notag
\end{align}
and by~\eqref{est-100}, this result immediately leads to the following estimate
\begin{align}\label{est_lem:B5a}
\|\nabla \tilde{\mathbf{w}}\|^p_{L^{\infty}(\tilde{B}_{\varrho/8}\cap \Omega;\mathbb{R}^{d^2})} + \|\mathrm{\pi}_{\tilde{\mathbf{w}}}\|^{p'}_{L^{\infty}(\tilde{B}_{\varrho/8}\cap \Omega)} & \le C \fint_{\Omega_{10}} \left(\sigma^2+|\nabla \mathbf{u}|^2\right)^{\frac{p}{2}} dx \notag \\
& \qquad + C\fint_{\Omega_{10}} \left(\sigma^p + |\mathbf{f}|^p + |\nabla \mathbf{g}|^p\right) dx. 
\end{align}
Then, due to the first relation~\eqref{cond:B2}, it enables us to derive
\begin{align*}
& \|\nabla \tilde{\mathbf{w}}\|^p_{L^{\infty}(\Omega_{1};\mathbb{R}^{d^2})} + \|\mathrm{\pi}_{\tilde{\mathbf{w}}}\|^{p'}_{L^{\infty}(\Omega_{1})} \le \|\nabla \tilde{\mathbf{w}}\|^p_{L^{\infty}(\tilde{B}_{\varrho/8}\cap \Omega;\mathbb{R}^{d^2})} + \|\mathrm{\pi}_{\tilde{\mathbf{w}}}\|^{p'}_{L^{\infty}(\tilde{B}_{\varrho/8}\cap \Omega)},
\end{align*}
which ensures the estimate in~\eqref{est-boundary-2} from~\eqref{est_lem:B5a}. On the other hand, by the same technique as the previous proofs, we also obtain the following comparison estimate 
\begin{align} \label{est_lem:B5b}
\fint_{\tilde{B}_{\varrho/8}\cap \Omega} |\nabla \tilde{\mathbf{v}} - \nabla \tilde{\mathbf{w}}|^p + |\mathrm{\pi}_{\tilde{\mathbf{v}}} - \mathrm{\pi}_{\tilde{\mathbf{w}}}|^{p'} dx & \le C\delta^{\kappa} \fint_{\tilde{B}_{\varrho/4}\cap \Omega} \left(\sigma^2 + |\nabla \mathbf{v}|^2\right)^{\frac{p}{2}} dx \notag \\
& \le C \delta^{\kappa} \fint_{\Omega_{10}} \left(\sigma^2+|\nabla \mathbf{u}|^2\right)^{\frac{p}{2}} dx  + C \delta^{\kappa}\fint_{\Omega_{10}} \left(\sigma^p + |\mathbf{f}|^p + |\nabla \mathbf{g}|^p\right) dx. 
\end{align} 
Finally, thanks to~\eqref{cond:B2}, we infer that
\begin{align} 
 \fint_{\Omega_{1}} |\nabla \mathbf{u} - \nabla \tilde{\mathbf{w}}|^p +  |\mathrm{\pi} - \mathrm{\pi}_{\tilde{\mathbf{w}}}|^{p'} dx &\le C \fint_{\Omega \cap \tilde{B}_{\varrho/8}} |\nabla \mathbf{u} - \nabla \tilde{\mathbf{w}}|^p + |\mathrm{\pi} - \mathrm{\pi}_{\tilde{\mathbf{w}}}|^{p'} dx \notag \\ 
&  \le C \fint_{\Omega \cap \tilde{B}_{\varrho/8}} |\nabla \mathbf{u} - \nabla \mathbf{v}|^p + |\mathrm{\pi} - \mathrm{\pi}_{\mathbf{v}}|^{p'} dx \notag \\
& \hspace{1cm}  + C \fint_{\Omega \cap \tilde{B}_{\varrho/8}} |\nabla \mathbf{v} - \nabla \tilde{\mathbf{v}}|^p + |\mathrm{\pi}_{\mathbf{v}} - \mathrm{\pi}_{\tilde{\mathbf{v}}}|^{p'} dx \notag \\    
& \hspace{2cm} + C \fint_{\Omega \cap \tilde{B}_{\varrho/8}} |\nabla \tilde{\mathbf{v}} - \nabla \tilde{\mathbf{w}}|^p + |\mathrm{\pi}_{\tilde{\mathbf{v}}} - \mathrm{\pi}_{\tilde{\mathbf{w}}}|^{p'} dx, \notag
\end{align}
And the desired result~\eqref{est-boundary-1} follows from combining~\eqref{est-100} with~\eqref{est_lem:v-w500} and~\eqref{est_lem:B5b}.
\end{proof}

\section{Global regularity estimates in generalized spaces}
\label{sec:proof}

Once having technical comparison results at hand, one may employ them to transfer the level-set estimate via WFMDs, as stated in Theorem~\ref{theo:dist}. Further, this section is also devoted to the proofs of our desired results: Theorem~\ref{theo:main} and Theorem~\ref{theo:improv}, which concern generalized weighted Lorentz and $\psi$-generalized Morrey regularity.

\subsection{Level-set inequality on FMDs in the Muckenhoupt weighted setting}

The idea to prove level-set inequality on FMDs dates back to~\cite{AM2007}, who directly used a version of Calder\'on-Zygmund coverings arguments. In this study, one can apply Calder\'on-Zygmund-Krylov-Safanov type covering lemma to suitable level sets of the fractional maximal functions. For the convenience of the reader, we only restate it in a way suitable to our needs, and we refer to~\cite{CP1998,KS1980} for the complete proof.
\begin{lemma}[Covering lemma]
\label{lem:VCL}
Let $\omega \in \mathcal{A}_\infty$ be a Muckenhoupt weight and $\mathcal{S}_1 \subset \mathcal{S}_2$ be two measurable subsets of $(\delta_0,R_0)$-Reifenberg flat domain $\Omega$ for some $\delta_0 \in (0,1)$ and $R_0>0$. Suppose that $N$ balls $\{B^{j}\}_{j=1}^N$ with radius $R_0$ covers $\Omega$. Assume moreover that there is $\varepsilon \in (0,1)$ such that:
\begin{itemize}
\item[(i)] $\omega(\mathcal{S}_1) \le \varepsilon \omega(B^{j})$ for all $j=1,2,...,N$;
\item[(ii)] for every $x_0 \in \Omega$ and $\varrho \in (0,R_0]$, if $\omega(B_\varrho(x_0)\cap\mathcal{S}_1) \ge \varepsilon\omega(B_\varrho(x_0))$ then $B_\varrho(x_0) \cap \Omega \subset \mathcal{S}_2$.
\end{itemize}
Then, there exists a constant $C = C(d,[\omega]_{\mathcal{A}_\infty})$ such that $\omega(\mathcal{S}_1) \le C\varepsilon \omega(\mathcal{S}_2)$.
\end{lemma}
\begin{lemma}\label{lem:D1}
Let $\{B^{j}\}_{j=1}^N$ be $N$ balls whose radius $R_0$ and centers belonging to $\partial\Omega$. Suppose that $\{B^{j}\}_{j=1}^N$ covers $\Omega$. Then for every $\mathfrak{a}>0$ and $\varepsilon>0$, there exists $\mathfrak{b}=\mathfrak{b}(\mathfrak{a},\varepsilon,\mathtt{data}_1)>0$ such that $\omega(\mathbb{S}(\mathfrak{a},\mathfrak{b})) \le \varepsilon \omega(B^j)$ for all $j= 1,2,...,N$, where $\mathbb{S}(\mathfrak{a},\mathfrak{b})$ is defined by
\begin{align}\label{def:S1}
\mathbb{S}(\mathfrak{a},\mathfrak{b}) := \left\{\mathcal{M}_{\alpha}\mathbb{U} > \mathfrak{a} \lambda; \ \mathcal{M}_{\alpha}\mathbb{F}_{\sigma} \le \mathfrak{b}\lambda\right\}.
\end{align}
\end{lemma}
\begin{proof}
Without loss of generality, we can reduce ourselves to the case where $\mathbb{S}(\mathfrak{a},\mathfrak{b})$ is a non-empty set. Then, there exists $x_* \in \Omega$ such that $\mathcal{M}_{\alpha}\mathbb{F}_{\sigma}(x_*) \le \mathfrak{b}\lambda$. Let $D_0 = \mathrm{diam}(\Omega)$ and $\mathcal{B} = B_{(2D_0+R_0)}(x_*)$, it is easy to check that
$$\mathbb{S}(\mathfrak{a},\mathfrak{b}) \subset \Omega \subset \mathcal{B} \quad \mbox{ and } \quad B^{j} \subset \mathcal{B}, \quad \mbox{ for all } j= 1,2,...,N.$$ 
Thanks to property~\eqref{ineq-Muck}, it yields that
\begin{align}\notag
\omega(\mathbb{S}(\mathfrak{a},\mathfrak{b})) \le C_2 \left(\frac{|\mathbb{S}(\mathfrak{a},\mathfrak{b})|}{|\mathcal{B}|}\right)^{\iota_2} \omega(\mathcal{B}), \quad \mbox{ and } \quad \omega(\mathcal{B}) \le C_1^{-1} \left(\frac{|\mathcal{B}|}{|B^j|}\right)^{\iota_1} \omega(B^j), 
\end{align}
for every $j= 1,2,...,N$. Moreover, note that $\displaystyle{\frac{|\mathcal{B}|}{|B^j|}} = \displaystyle{\frac{2D_0+R_0}{R_0}}$, and due to this fact, we readily infer
\begin{align}\label{est-002}
\omega(\mathbb{S}(\mathfrak{a},\mathfrak{b})) \le C_1^{-1} C_2 \left(\frac{2D_0}{R_0}+1\right)^{\iota_1} \left(\frac{|\mathbb{S}(\mathfrak{a},\mathfrak{b})|}{|\mathcal{B}|}\right)^{\iota_2}  \omega(B^j). 
\end{align}
In order to bound the ratio $\displaystyle{\frac{|\mathbb{S}(\mathfrak{a},\mathfrak{b})|}{|\mathcal{B}|}}$, we will make use of the boundedness property of $\mathcal{M}_\alpha$ in Lemma~\ref{bound-M-beta} and the global estimate proved in Lemma~\ref{lem:global}. In particular, it follows that
\begin{align}\label{est-001}
|\mathbb{S}(\mathfrak{a},\mathfrak{b})| \le C \left(\frac{1}{\mathfrak{a}\lambda}\int_{\Omega}\mathbb{U}(x)dx\right)^{\frac{d}{d-\alpha}} \le C \left(\frac{1}{\mathfrak{a}\lambda}\int_{\Omega}\mathbb{F}_{\sigma}(x)dx\right)^{\frac{d}{d-\alpha}}.
\end{align}
One the other hand, exploiting the fact that $\Omega \subset \mathcal{B}$ and using $\mathcal{M}_{\alpha}\mathbb{F}_{\sigma}(x_*) \le \mathfrak{b}\lambda$, it yields
\begin{align*}
\int_{\Omega}\mathbb{F}_{\sigma}(x)dx & \le \int_{\mathcal{B}}\mathbb{F}_{\sigma}(x)dx  \le |\mathcal{B}| (2D_0+R_0)^{-\alpha} \mathcal{M}_{\alpha}\mathbb{F}_{\sigma}(x_*) \le C |\mathcal{B}|^{1-\frac{\alpha}{d}} \mathfrak{b}\lambda.
\end{align*} 
And from~\eqref{est-001}, we readily obtain $|\mathbb{S}(\mathfrak{a},\mathfrak{b})| \le C (\mathfrak{b}/\mathfrak{a})^{\frac{d}{d-\alpha}} |\mathcal{B}|$. Combining this with~\eqref{est-002}, we conclude
\begin{align}\label{est-003}
\omega(\mathbb{S}(\mathfrak{a},\mathfrak{b})) \le C \left(\frac{2D_0}{R_0}+1\right)^{\iota_1} (\mathfrak{b}/\mathfrak{a})^{\frac{d\iota_2}{d-\alpha}} \omega(B^j), 
\end{align}
for each $j = 1,2,...,N$. Moreover, for every $\mathfrak{a}>0$ and $\varepsilon>0$, we observe that this time we can choose $\mathfrak{b}=\mathfrak{b}(\mathfrak{a},\varepsilon,\mathtt{data}_1)>0$ small enough such that 
\begin{align}\notag 
C \left(\frac{2D_0}{R_0}+1\right)^{\iota_1} (\mathfrak{b}/\mathfrak{a})^{\frac{d\iota_2}{d-\alpha}} < \varepsilon, 
\end{align}
and the previous inequality~\eqref{est-003} directly implies $\omega(\mathbb{S}(\mathfrak{a},\mathfrak{b})) \le \varepsilon \omega(B^j)$.
\end{proof}

\begin{lemma}\label{lem:D2}
Let $x_0 \in \Omega$, $\varrho \in (0,R_0]$ and assume that there exists $x_1 \in B_{\varrho}(x_0)$ such that $\mathcal{M}_{\alpha}\mathbb{U}(x_1) \le \lambda$ for some $\lambda>0$. Then, for all $\mathfrak{a}>3^d$, the following estimate holds:
\begin{align}\label{est-lemD2}
|B_\varrho(x_0)\cap \mathbb{S}(\mathfrak{a},\mathfrak{b})| \le |B_\varrho(x_0)\cap \{\mathcal{M}_{\alpha}^{\varrho}(\chi_{B_{2\varrho}(x_0)}\mathbb{U}) > \mathfrak{a}\lambda\}|,
\end{align}
where $\mathcal{M}_{\alpha}^{\varrho}$ is so-called the cut-off fractional maximal operator, defined as
\begin{align*}
& {\mathcal{M}}^{\varrho}_{\alpha}\mathsf{g}(x)  = \sup_{0<r<\varrho} r^{\alpha} \fint_{B_r(x)} |\mathsf{g}(\nu)|d\nu, \quad x \in \mathbb{R}^d, \ \mathsf{g} \in L^1_{\mathrm{loc}}(\mathbb{R}^d). 
\end{align*}
\end{lemma}
\begin{proof}
First, we fix a point $y \in B_\varrho(x_0)$. Then, for every $r \ge \varrho$ and $z \in B_r(y)$, it is clear that
\begin{align}\notag
|z-x_1| \le |z-y| + |y-x_0| + |x_0-x_1| < r + \varrho + \varrho < 3r,
\end{align}
which directly implies $B_r(y) \subset B_{3r}(x_1)$. It follows that
\begin{align}\notag\mathcal{T}_{\alpha}^{\varrho}\mathbb{U}(y) := \sup_{r\ge \varrho} r^{\alpha} \fint_{B_r(y)} \mathbb{U}(z) dz \le 3^d \sup_{r\ge \varrho} r^{\alpha} \fint_{B_{3r}(x_1)} \mathbb{U}(z) dz \le 3^{d-\alpha} \mathcal{M}_{\alpha}\mathbb{U}(x_1) \le 3^d \lambda.
\end{align}
Further, it also allows us to get 
$$ \mathcal{M}_{\alpha}\mathbb{U}(y) \le \max\left\{\mathcal{M}_{\alpha}^{\varrho}\mathbb{U}(y); \, 3^d \lambda\right\}.$$ 
On the other hand, for every $r < \varrho$ one also has $B_r(y) \subset B_{2\varrho}(x_0)$ and therefore
\begin{align}\notag
\mathcal{M}_{\alpha}^{\varrho}\mathbb{U}(y) = \sup_{0<r<\varrho} r^{\alpha} \fint_{B_r(y)} \mathbb{U}(z) dz = \sup_{0<r<\varrho} r^{\alpha} \fint_{B_r(y)} \left(\chi_{B_{2\varrho}(x_0)} \mathbb{U}\right)(z) dz.
\end{align}
Due to the reasons just shown above, it leads to the desired result
\begin{align}\notag
B_\varrho(x_0)\cap \mathbb{S}(\mathfrak{a},\mathfrak{b}) \subset B_\varrho(x_0)\cap \{\mathcal{M}_{\alpha}^{\varrho}\mathbb{U} > \mathfrak{a}\lambda\} = B_\varrho(x_0)\cap \{\mathcal{M}_{\alpha}^{\varrho}(\chi_{B_{2\varrho}(x_0)}\mathbb{U}) > \mathfrak{a}\lambda\}.
\end{align}
for all $\mathfrak{a} >3^d$.
\end{proof}

\begin{lemma}\label{lem:D3}
Assume that $B_\varrho(x_0) \cap \Omega \not\subset \left\{\mathcal{M}_{\alpha}\mathbb{U} > \lambda\right\}$ for some $x_0 \in \Omega$ and $\varrho \in (0,R_0]$. 
Then one can find a constant $\mathfrak{a} = \mathfrak{a}(\mathtt{data}_1)>0$ such that: for every $\varepsilon \in (0,1)$, there exist $\delta_0=\delta_0(\varepsilon,\mathtt{data}_1) \in (0,1/81)$ and $\mathfrak{b} = \mathfrak{b}(\varepsilon,\mathtt{data}_1)>0$ satisfying
\begin{align}\label{est-lemD3}
\omega(B_\varrho(x_0)\cap\mathbb{S}(\mathfrak{a},\mathfrak{b})) < \varepsilon\omega(B_\varrho(x_0)),
\end{align}
if provided $(\Omega,\mathbf{A}) \in \mathcal{H}(\delta_0,R_0)$. Here, the set $\mathbb{S}(\mathfrak{a},\mathfrak{b})$ was defined in Lemma~\ref{lem:D1}.
\end{lemma}
\begin{proof}
We may assume with no loss of generality that $B_\varrho(x_0)\cap\mathbb{S}(\mathfrak{a},\mathfrak{b})\neq \emptyset$. Moreover, since $B_\varrho(x_0) \cap \Omega \not\subset \left\{\mathcal{M}_{\alpha}\mathbb{U} > \lambda\right\}$, it is therefore natural to find out $x_1, x_2 \in B_\varrho(x_0) \cap \Omega$ such that 
\begin{align}\label{x1x2}
\mathcal{M}_{\alpha}\mathbb{U}(x_1) \le \lambda, \quad \mathcal{M}_{\alpha}\mathbb{F}_{\sigma}(x_2) \le \mathfrak{b} \lambda.
\end{align}
Note that for $x_0 \in \Omega$, it can be viewed as two separate cases: $B_{8\varrho}(x_0) \subset \Omega$ and $B_{8\varrho}(x_0) \cap \partial \Omega \neq \emptyset$. For the second one, we claim that there exists $x_3 \in \partial \Omega$ such that 
$$\mathrm{dist}(x_0,\partial \Omega) = |x_0-x_3| < 8\varrho.$$ 
For future notational convenience, we regard $\Omega_1 = B_R(x_b) \cap \Omega$ and $\Omega_{k} = B_{kR}(x_b) \cap \Omega$, where
\begin{align}\notag
B_R(x_b) = \begin{cases} B_{8\varrho}(x_0), &\mbox{ if } B_{8\varrho}(x_0) \subset \Omega, \\ B_{12\varrho}(x_3), &\mbox { if } B_{8\varrho}(x_0) \cap \partial \Omega \neq \emptyset, \end{cases}
\end{align}
and
\begin{align}\notag
k = \begin{cases} 4, &\mbox{ if } B_{8\varrho}(x_0) \subset \Omega, \\ 10, &\mbox { if } B_{8\varrho}(x_0) \cap \partial \Omega \neq \emptyset. \end{cases}
\end{align}
Invoking Lemma~\ref{lem:comp-2} and Lemma~\ref{lem:boundary}, one can find $(\mathbf{w},\mathrm{\pi}_{\mathbf{w}}) \in W^{1,p}(\Omega_1;\mathbb{R}^d) \times L^{p'}(\Omega_1)$ such that
\begin{align}\label{est-42}
 \|\nabla \mathbf{w}\|_{L^{\infty}(\Omega_1;\mathbb{R}^{d^2})}^p + \|\mathrm{\pi}_{\mathbf{w}}\|_{L^{\infty}(\Omega_1)}^{p'} &\le C\fint_{\Omega_{k}} \mathbb{U}(x) + \mathbb{F}_{\sigma}(x) dx,
\end{align}
and
\begin{align}\label{est-41}
 \fint_{\Omega_1}|\nabla \mathbf{u} - \nabla \mathbf{w}|^p + |\mathrm{\pi} - \mathrm{\pi}_{\mathbf{w}}|^{p'} dx  & \le C \delta^{\kappa} \fint_{\Omega_{k}} \mathbb{U}(x) dx  + C \delta^{-\overline{p}\kappa} \fint_{\Omega_{k}} \mathbb{F}_{\sigma}(x) dx,
\end{align}
whenever $[\mathbf{A}]_{R_0} \le \delta$ for some $\delta \in (0,1)$, where the constant $C = C(\Omega,d,\Upsilon)>0$. Here, $\kappa$ is the positive constant appearing in Lemma~\ref{lem:comp-2}. Arriving at this stage, one can easily check that $B_{2\varrho}(x_0) \subset B_R(x_b)$, which therefore allows us to rewrite inequality~\eqref{est-lemD2} in Lemma~\ref{lem:D2} as
\begin{align}\notag
|B_\varrho(x_0)\cap \mathbb{S}(\mathfrak{a},\mathfrak{b})| \le |B_\varrho(x_0)\cap \{\mathcal{M}_{\alpha}^{\varrho}(\chi_{\Omega_1}\mathbb{U}) > \mathfrak{a}\lambda\}|.
\end{align}
for $\mathfrak{a}>3^d$. At this point, we are able to estimate
\begin{align}\label{est-007}
|B_\varrho(x_0)\cap \mathbb{S}(\mathfrak{a},\mathfrak{b})| & \le |B_\varrho(x_0)\cap \{\mathcal{M}_{\alpha}^{\varrho}(\chi_{\Omega_1}|\nabla \mathbf{u}|^p)> 2^{-1}\mathfrak{a}\lambda\}| \notag\\
& \hspace{1cm} + |B_\varrho(x_0)\cap \{\mathcal{M}_{\alpha}^{\varrho}(\chi_{\Omega_1}|\mathrm{\pi}|^{p'}) > 2^{-1}\mathfrak{a}\lambda\}| \notag \\
& \le |B_\varrho(x_0)\cap \{\mathcal{M}_{\alpha}^{\varrho}(\chi_{\Omega_1}|\nabla \mathbf{u} - \nabla \mathbf{w}|^p)> 2^{-p-1}\mathfrak{a}\lambda\}| \notag\\
& \hspace{1cm} + |B_\varrho(x_0)\cap \{\mathcal{M}_{\alpha}^{\varrho}(\chi_{\Omega_1}|\mathrm{\pi}-\mathrm{\pi}_{\mathbf{w}}|^{p'}) > 2^{-p'-1}\mathfrak{a}\lambda\}| \notag \\
& \hspace{2cm} + |B_\varrho(x_0)\cap \{\mathcal{M}_{\alpha}^{\varrho}(\chi_{\Omega_1}|\nabla \mathbf{w}|^p)> 2^{-p-1}\mathfrak{a}\lambda\}| \notag\\
& \hspace{3cm} + |B_\varrho(x_0)\cap \{\mathcal{M}_{\alpha}^{\varrho}(\chi_{\Omega_1}|\mathrm{\pi}_{\mathbf{w}}|^{p'}) > 2^{-p'-1}\mathfrak{a}\lambda\}|.
\end{align}
Since $x_1 \in B_{\varrho}(x_0)$, the ball $B_{kR}(x_b)$ can be covered by ball $B_{(k+1)R}(x_1)$. Thus, a combination of this assertion and~\eqref{x1x2} yields that
\begin{align}\label{est-43}
\fint_{\Omega_k}\mathbb{U}(x)dx \le \frac{|B_{(k+1)R}(x_1)|}{|B_{kR}(x_b)|} \fint_{B_{(k+1)R}(x_1)}\mathbb{U}(x)dx \le C \varrho^{-\alpha} \mathcal{M}_{\alpha} \mathbb{U}(x_1) \le C \varrho^{-\alpha} \lambda.
\end{align}
Follow the similar path, we also obtain
\begin{align}\label{est-43-b}
\fint_{\Omega_k}\mathbb{F}_{\sigma}(x)dx \le \frac{|B_{(k+1)R}(x_2)|}{|B_{kR}(x_b)|} \fint_{B_{(k+1)R}(x_2)}\mathbb{F}_{\sigma}(x)dx \le C \varrho^{-\alpha} \mathcal{M}_{\alpha} \mathbb{F}_{\sigma}(x_2) \le C \varrho^{-\alpha} \mathfrak{b} \lambda.
\end{align}
Substituting~\eqref{est-43} and~\eqref{est-43-b} into~\eqref{est-42}, it implies that
\begin{align}\label{est-45}
& \|\nabla \mathbf{w}\|_{L^{\infty}(\Omega_1;\mathbb{R}^{d^2})}^p + \|\mathrm{\pi}_{\mathbf{w}}\|_{L^{\infty}(\Omega_1)}^{p'} \le C \varrho^{-\alpha} (1 + \mathfrak{b}) \lambda  \le C^* \varrho^{-\alpha} \lambda.
\end{align}
By~\eqref{est-45}, for every $y \in B_{\varrho}(x)$, the following estimate holds true:
\begin{align}\notag
& \mathcal{M}_{\alpha}^{\varrho}(\chi_{\Omega_1}|\nabla \mathbf{w}|^p)(y) + \mathcal{M}_{\alpha}^{\varrho}(\chi_{\Omega_1}|\mathrm{\pi}_{\mathbf{w}}|^{p'})(y) \le \varrho^{\alpha} \left(\|\nabla \mathbf{w}\|_{L^{\infty}(\Omega_1;\mathbb{R}^{d^2})}^p + \|\mathrm{\pi}_{\mathbf{w}}\|_{L^{\infty}(\Omega_1)}^{p'}\right) \le C^* \lambda,
\end{align}
which guarantees that two sets 
\begin{align}\notag
B_\varrho(x_0)\cap \{\mathcal{M}_{\alpha}^{\varrho}(\chi_{\Omega_1}|\nabla \mathbf{w}|^p)> 2^{-p-1}\mathfrak{a}\lambda\} \ \mbox{ and } \ B_\varrho(x_0)\cap \{\mathcal{M}_{\alpha}^{\varrho}(\chi_{\Omega_1}|\mathrm{\pi}_{\mathbf{w}}|^{p'}) > 2^{-p'-1}\mathfrak{a}\lambda\}
\end{align}
are empty if chosen $\mathfrak{a}> \max\{3^d; 2^{p+1}C^*, 2^{p'+1}C^*\}$, where $C^*$ is exactly the constant pointed out in~\eqref{est-45}. From~\eqref{est-007} and Lemma~\ref{bound-M-beta}, we arrive at
\begin{align}\label{est-008}
|B_\varrho(x_0)\cap \mathbb{S}(\mathfrak{a},\mathfrak{b})| & \le |B_\varrho(x_0)\cap \{\mathcal{M}_{\alpha}^{\varrho}(\chi_{\Omega_1}|\nabla \mathbf{u} - \nabla \mathbf{w}|^p)> 2^{-p-1}\mathfrak{a}\lambda\}| \notag\\
& \hspace{1cm} + |B_\varrho(x_0)\cap \{\mathcal{M}_{\alpha}^{\varrho}(\chi_{\Omega_1}|\mathrm{\pi}-\mathrm{\pi}_{\mathbf{w}}|^{p'}) > 2^{-p'-1}\mathfrak{a}\lambda\}| \notag \\
& \le C\left(\frac{1}{2^{-p-1}\mathfrak{a}\lambda}\int_{\Omega_1}|\nabla \mathbf{u} - \nabla \mathbf{w}|^pdx\right)^{\frac{d}{d-\alpha}} \notag \\
& \hspace{2cm} + C\left(\frac{1}{2^{-p'-1}\mathfrak{a}\lambda}\int_{\Omega_1}|\mathrm{\pi}-\mathrm{\pi}_{\mathbf{w}}|^{p'} dx\right)^{\frac{d}{d-\alpha}}\notag \\
& \le  C\left(\frac{\varrho^d}{\mathfrak{a}\lambda}\fint_{\Omega_1}|\nabla \mathbf{u} - \nabla \mathbf{w}|^p + |\mathrm{\pi}-\mathrm{\pi}_{\mathbf{w}}|^{p'} dx\right)^{\frac{d}{d-\alpha}}.
\end{align}
Next, combining the preceding estimates~\eqref{est-43},~\eqref{est-43-b} and taking~\eqref{est-41} into account, we obtain
\begin{align}\notag 
& \fint_{\Omega_1}|\nabla \mathbf{u} - \nabla \mathbf{w}|^p + |\mathrm{\pi} - \mathrm{\pi}_{\mathbf{w}}|^{p'} dx \le C \varrho^{-\alpha} \left(\delta^{\kappa} + \delta^{-\overline{p}\kappa}\mathfrak{b}\right) \lambda. 
\end{align}
Then, invoking~\eqref{est-008}, we estimate
\begin{align}\label{est-009}
|B_\varrho(x_0)\cap \mathbb{S}(\mathfrak{a},\mathfrak{b})| & \le C\left(\frac{\varrho^d}{\mathfrak{a}\lambda} \varrho^{-\alpha} \left(\delta^{\kappa} + \delta^{-\overline{p}\kappa}\mathfrak{b}\right) \lambda \right)^{\frac{d}{d-\alpha}} \notag \\
& \le C \left(\frac{\delta^{\kappa} + \delta^{-\overline{p}\kappa}\mathfrak{b}}{\mathfrak{a}}\right)^{\frac{d}{d-\alpha}} |B_\varrho(x_0)|.
\end{align}
Consequently, we exploit the weight $\omega \in \mathcal{A}_{\infty}$ and make use of~\eqref{est-009} to conclude
\begin{align}\label{est-010}
\omega(B_\varrho(x_0)\cap \mathbb{S}(\mathfrak{a},\mathfrak{b})) & \le C_2 \left(\frac{|B_\varrho(x_0)\cap \mathbb{S}(\mathfrak{a},\mathfrak{b})|}{|B_\varrho(x_0)|}\right)^{\iota_2} \omega(B_\varrho(x_0)) \notag \\
& \le C^{**} \left(\frac{\delta^{\kappa} + \delta^{-\overline{p}\kappa}\mathfrak{b}}{\mathfrak{a}}\right)^{\frac{d\iota_2}{d-\alpha}} \omega(B_\varrho(x_0)) .
\end{align}
Finally, for every $\varepsilon >0$ and for $\mathfrak{a} = \mathfrak{a}(\mathtt{data}_1)>0$ large enough, at this stage we choose $\delta=\delta_0(\varepsilon,\mathtt{data}_1) \in (0,1/81)$ and $\mathfrak{b} = \mathfrak{b}(\varepsilon,\mathtt{data}_1)>0$ such that 
$$C^{**} \left(\frac{\delta_0^{\kappa} + \delta_0^{-\overline{p}\kappa}\mathfrak{b}}{\mathfrak{a}}\right)^{\frac{d\iota_2}{d-\alpha}} < \varepsilon,$$
and this completes the proof of the Lemma.
\end{proof}

With all the preceding results at hand, the proof of  Theorem~\ref{theo:dist} is quite simple. 

\begin{proof}[Proof of Theorem~\ref{theo:dist}]
Let us first consider two subsets of $\Omega$ as following
\begin{align*}
\mathcal{S}_1 :=\mathbb{S}(\mathfrak{a},\mathfrak{b}) = \left\{\mathcal{M}_{\alpha}\mathbb{U} > \mathfrak{a} \lambda; \ \mathcal{M}_{\alpha}\mathbb{F}_{\sigma} \le \mathfrak{b}\lambda\right\}, \ \mbox{ and } \ \mathcal{S}_2 := \left\{\mathcal{M}_{\alpha}\mathbb{U} >  \lambda\right\}.
\end{align*}
Thanks to Lemma~\ref{lem:D1} and Lemma~\ref{lem:D3}, it is possible to find a constant $\mathfrak{a} = \mathfrak{a}(\mathtt{data}_1)>0$ such that for every $\varepsilon \in (0,1)$, the statement of $(i)$ and $(ii)$ in Lemma~\ref{lem:VCL} hold true with appropriate choice of $\delta_0=\delta_0(\varepsilon,\mathtt{data}_1)>0$ and $\mathfrak{b} = \mathfrak{b}(\varepsilon,\mathtt{data}_1)>0$, where $(\Omega,\mathbf{A}) \in \mathcal{H}(\delta_0,R_0)$ for some $R_0>0$. This finally gives
\begin{align*}
\omega\left(\left\{\mathcal{M}_{\alpha}\mathbb{U} > \mathfrak{a} \lambda; \ \mathcal{M}_{\alpha}\mathbb{F}_{\sigma} \le \mathfrak{b}\lambda\right\}\right) \le C \varepsilon \omega\left(\left\{\mathcal{M}_{\alpha}\mathbb{U} > \lambda\right\}\right),
\end{align*}
and directly implies to the following level-set inequality
\begin{align*}
\omega\left(\left\{\mathcal{M}_{\alpha}\mathbb{U} > \mathfrak{a} \lambda\right\}\right) \le C \varepsilon \omega\left(\left\{\mathcal{M}_{\alpha}\mathbb{U} > \lambda\right\}\right) + \omega\left(\left\{\mathcal{M}_{\alpha}\mathbb{F}_{\sigma} > \mathfrak{b}\lambda\right\}\right).
\end{align*}
This is equivalent to the desired estimate~\eqref{ineq-dist}. The proof is complete.
\end{proof}

\subsection{In generalized weighted Lorentz spaces}

At this point, having at hand the key ingredient: level-set inequality via WFMDs, the target is clearly defined. The aim of this section is to verify regularity estimates in the generalized function spaces as outlined in Section~\ref{sec:intro}. The first result we want to mention here is the global regularity in Lorentz spaces with two weights. 

\begin{proof}[Proof of Theorem~\ref{theo:main}]
In the framework of Theorem~\ref{theo:dist}, it leads us to find a constant $\mathfrak{a} = \mathfrak{a}(\mathtt{data}_1)>0$ such that for every $\varepsilon \in (0,1)$, there exist $\delta_0=\delta_0(\varepsilon,\mathtt{data}_1) \in (0,1/81)$, $\mathfrak{b} = \mathfrak{b}(\varepsilon,\mathtt{data}_1)>0$ and $C^*=C^*(\mathtt{data}_1)>1$ satisfying
\begin{align}\label{est-111}
\mathbb{X}(\mathfrak{a}\lambda) \le C^* \left[\varepsilon \mathbb{X}(\lambda) +  \mathbb{Y}(\mathfrak{b}\lambda)\right], \quad \forall \lambda>0,
\end{align}
when the assumption $(\Omega,\mathbf{A}) \in \mathcal{H}(\delta_0,R_0)$ is imposed, for some $R_0>0$. Taking the non-decreasing property of $\Psi$ and assumption~\eqref{cond:V} into account, for every $\lambda_1, \lambda_2 \ge 0$, we have
\begin{align}\notag 
\Psi(\lambda_1 + \lambda_2) \le \Psi(2 \max\{\lambda_1; \lambda_2\}) \le \beta_2 \Psi( \max\{\lambda_1; \lambda_2\}) \le \beta_2 (\Psi(\lambda_1) + \Psi(\lambda_2)). 
\end{align}
Now, it is clear to claim that $m = [\log_2(C^*)] \in \mathbb{N}$ (the greatest integer less than $\log_2(C^*)$). And thus, we readily deduce $2^{m-1} < C^* \le 2^m$. By~\eqref{Del-2}, for every $\lambda>0$ one has
\begin{align}\notag 
\Psi(C^*\lambda) \le \Psi(2^m\lambda) \le \beta_2^m \Psi(\lambda).
\end{align}
Combining the last two inequalities and merging them with~\eqref{est-111}, we arrive at
\begin{align}\label{est-112b}
\Psi\left(\mathbb{X}(\mathfrak{a}\lambda)\right) & \le \beta_2 \left[  \Psi\left(C^*\varepsilon \mathbb{X}(\lambda)\right) + \Psi\left(C^* \mathbb{Y}(\mathfrak{b}\lambda)\right) \right] \le \beta_2^{m+1} \left[\Psi\left(\varepsilon \mathbb{X}(\lambda)\right) + \Psi\left(\mathbb{Y}(\mathfrak{b}\lambda)\right) \right].
\end{align}
At this stage, for all $\mathfrak{s},\mathfrak{t} \in (0,\infty)$, the quasi-norm of generalized weighted Lorentz spaces in Definition~\ref{def:Lorentz} rewrites
\begin{align}\label{est-113}
\|\mathcal{M}_{\alpha}\mathbb{U}\|_{\mathcal{L}^{\mathfrak{s},\mathfrak{t}}_{\mu,\omega}(\Omega)}^{\mathfrak{t}} & = \mathfrak{s} \displaystyle{\int_0^\infty \lambda^{\mathfrak{t}} \left[\Psi\left(\mathbb{X}(\lambda)\right) \right]^{\frac{\mathfrak{t}}{\mathfrak{s}}} \frac{d\lambda}{\lambda}}  = {\mathfrak{a}}^{\mathfrak{t}} \mathfrak{s} \displaystyle{\int_0^\infty \lambda^{\mathfrak{t}} \left[\Psi\left(\mathbb{X}(\mathfrak{a}\lambda)\right) \right]^{\frac{\mathfrak{t}}{\mathfrak{s}}} \frac{d\lambda}{\lambda}}.
\end{align}
Applying~\eqref{est-112b} to~\eqref{est-113}, we get
\begin{align}\label{est-114}
\|\mathcal{M}_{\alpha}\mathbb{U}\|_{\mathcal{L}^{\mathfrak{s},\mathfrak{t}}_{\mu,\omega}(\Omega)}^{\mathfrak{t}} & \le C {\mathfrak{a}}^{\mathfrak{t}} \mathfrak{s} \displaystyle{\int_0^\infty \lambda^{\mathfrak{t}} \left[\Psi\left(\varepsilon \mathbb{X}(\lambda)\right)\right]^{\frac{\mathfrak{t}}{\mathfrak{s}}} \frac{d\lambda}{\lambda}} \notag \\
& \qquad \qquad + C {\mathfrak{a}}^{\mathfrak{t}} \mathfrak{s} \displaystyle{\int_0^\infty \lambda^{\mathfrak{t}} \left[\Psi\left(\mathbb{Y}(\mathfrak{b}\lambda)\right)\right]^{\frac{\mathfrak{t}}{\mathfrak{s}}} \frac{d\lambda}{\lambda}}.
\end{align}
For every $\varepsilon \in (0,1)$, it finds $k_{\varepsilon} \in \mathbb{N}$ such that 
$$1/2<2^{k_{\varepsilon}} \varepsilon \le 1 \Leftrightarrow \log_2(1/\varepsilon)-1 < k_{\varepsilon} \le \log_2(1/\varepsilon),$$ 
and this allows us to make use of~\eqref{cond:V} to arrive at
\begin{align}\notag
 \Psi(\varepsilon\lambda) \le {\beta_1}^{-k_{\varepsilon}} \Psi(2^{k_{\varepsilon}}\varepsilon\lambda) \le {\beta_1}^{-k_{\varepsilon}} \Psi(\lambda) \le {\beta_1}^{1-\log_2(1/\varepsilon)} \Psi(\lambda), \quad \forall \lambda \ge 0.
\end{align}
Merging this inequality to~\eqref{est-114}, it further deduces that
\begin{align}\notag 
\|\mathcal{M}_{\alpha}\mathbb{U}\|_{\mathcal{L}^{\mathfrak{s},\mathfrak{t}}_{\mu,\omega}(\Omega)}^{\mathfrak{t}} & \le C {\mathfrak{a}}^{\mathfrak{t}}  {\beta_1}^{\frac{\mathfrak{t}}{\mathfrak{s}}(1-\log_2(1/\varepsilon))} \mathfrak{s}\displaystyle{\int_0^\infty \lambda^{\mathfrak{t}} \left[\Psi\left(\mathbb{X}(\lambda)\right)\right]^{\frac{\mathfrak{t}}{\mathfrak{s}}} \frac{d\lambda}{\lambda}} \notag \\
& \qquad \qquad + C {\mathfrak{a}}^{\mathfrak{t}} {\mathfrak{b}}^{-\mathfrak{t}} \mathfrak{s} \displaystyle{\int_0^\infty \lambda^{\mathfrak{t}} \left[\Psi\left(\mathbb{Y}(\lambda)\right)\right]^{\frac{\mathfrak{t}}{\mathfrak{s}}} \frac{d\lambda}{\lambda}}. \notag
\end{align}
This guarantees the following estimate
\begin{align}\label{est-116}
\|\mathcal{M}_{\alpha}\mathbb{U}\|_{\mathcal{L}^{\mathfrak{s},\mathfrak{t}}_{\mu,\omega}(\Omega)} & \le C {\mathfrak{a}}  {\beta_1}^{\frac{1-\log_2(1/\varepsilon)}{\mathfrak{s}}} \|\mathcal{M}_{\alpha}\mathbb{U}\|_{\mathcal{L}^{\mathfrak{s},\mathfrak{t}}_{\mu,\omega}(\Omega)} + C {\mathfrak{a}} {\mathfrak{b}}^{-1} \|\mathcal{M}_{\alpha}\mathbb{F}_{\sigma}\|_{\mathcal{L}^{\mathfrak{s},\mathfrak{t}}_{\mu,\omega}(\Omega)}.
\end{align}
In a similar fashion,~\eqref{est-116} also holds for $\mathfrak{t} = \infty$ and $\mathfrak{s} \in (0,\infty)$. At this point, note that by $\beta_1>1$, we have
$$ \displaystyle{\lim_{\varepsilon \to 0^+} {\beta_1}^{\frac{1-\log_2(1/\varepsilon)}{\mathfrak{s}}} = 0}.$$
Thus, to conclude we choose $\varepsilon = \varepsilon(\mathtt{data}_2)>0$ in~\eqref{est-116} small enough such that 
$$\displaystyle{C {\mathfrak{a}}  {\beta_1}^{\frac{1-\log_2(1/\varepsilon)}{\mathfrak{s}}} < \frac{1}{2}},$$ 
and the desired estimate~\eqref{ineq-main} follows easily.
\end{proof}

\subsection{In $\psi$-generalized Morrey spaces}

The second main outcome stated in Theorem~\ref{theo:improv} follows as a consequence of Theorem~\ref{theo:main}. Now we are in position to prove this result.

\begin{proof}[Proof of Theorem~\ref{theo:improv}]
Thanks to Theorem~\ref{theo:dist}, one can find a constant $\mathfrak{a} = \mathfrak{a}(\mathtt{data})>0$ such that for every $\varepsilon \in (0,1)$, there exist $\delta_0=\delta_0(\varepsilon,\mathtt{data}_1) \in (0,1/81)$, $\mathfrak{b} = \mathfrak{b}(\varepsilon,\mathtt{data}_1)>0$ and $C^*=C^*(\mathtt{data}_1)>1$ satisfying~\eqref{est-111} if $(\Omega,\mathbf{A}) \in \mathcal{H}(\delta_0,R_0)$ for some $R_0>0$. For every $\mathfrak{s} \in (0, \infty)$ and $\omega \in \mathcal{A}_{\infty}$, thanks to~\eqref{est-111} there holds
\begin{align}\label{est-200}
\|\mathcal{M}_{\alpha}\mathbb{U}\|_{L^{\mathfrak{s}}_{\omega}(\Omega)} &= {\mathfrak{a}} \left(\mathfrak{s} \displaystyle{\int_0^\infty \lambda^{\mathfrak{s}-1} \mathbb{X}(\mathfrak{a}\lambda) d\lambda}\right)^{\frac{1}{\mathfrak{s}}} \notag \\
& \le C  {\mathfrak{a}} \left(\varepsilon  \mathfrak{s} \displaystyle{\int_0^\infty \lambda^{\mathfrak{s}-1} \mathbb{X}(\lambda) d\lambda}\right)^{\frac{1}{\mathfrak{s}}} + C {\mathfrak{a}} \mathfrak{b}^{-1} \left( \mathfrak{s} \displaystyle{\int_0^\infty \lambda^{\mathfrak{s}-1} \mathbb{Y}(\lambda) d\lambda}\right)^{\frac{1}{\mathfrak{s}}} \notag \\
& = C  {\mathfrak{a}} \varepsilon^{\frac{1}{\mathfrak{s}}} \|\mathcal{M}_{\alpha}\mathbb{U}\|_{L^{\mathfrak{s}}_{\omega}(\Omega)} + C {\mathfrak{a}} \mathfrak{b}^{-1}\|\mathcal{M}_{\alpha}\mathbb{F}_{\sigma}\|_{L^{\mathfrak{s}}_{\omega}(\Omega)}.
\end{align}
Taking a suitable choice of $\varepsilon>0$ in~\eqref{est-200}, it yields
\begin{align}\notag
\int_{\Omega} \left[\mathcal{M}_{\alpha}\mathbb{U}(x)\right]^{\mathfrak{s}} \omega(x) dx & = \int_{\Omega} \left[\mathcal{M}_{\alpha}\mathbb{F}_{\sigma}(x)\right]^{\mathfrak{s}} \omega(x) dx.
\end{align}
Clearly, one may rewrite this inequality as
\begin{align}\label{est-0}
\int_{\mathbb{R}^d} \left[\chi_{\Omega}\mathcal{M}_{\alpha}\mathbb{U}(x)\right]^{\mathfrak{s}} \omega(x) dx \le C\int_{\mathbb{R}^d} \left[\chi_{\Omega}\mathcal{M}_{\alpha}\mathbb{F}_{\sigma}(x)\right]^{\mathfrak{s}} \omega(x) dx,
\end{align}
and at this stage, thanks to~\eqref{Morrey-norm}, we arrive at
\begin{align}\notag
\|\mathcal{M}_{\alpha}\mathbb{U}\|_{\mathrm{M}^{\mathfrak{s},\psi}(\Omega)} = \sup_{y\in \Omega; \, 0<\varrho<\mathrm{diam}(\Omega)} \left(\frac{1}{\psi(y,\varrho)}\int_{\Omega_{\varrho}(y)}|\mathcal{M}_{\alpha}\mathbb{U}(x)|^{\mathfrak{s}} dx\right)^{\frac{1}{\mathfrak{s}}}.
\end{align}
For every $y \in \Omega$ and $0<\varrho<\mathrm{diam}(\Omega)$, we have to estimate the following term
$$ T(y,\varrho) := \frac{1}{\psi(y,\varrho)}\int_{\Omega_{\varrho}(y)}[\mathcal{M}_{\alpha}\mathbb{U}(x)]^sdx.$$
For the sake of brevity, we shall denote $\mathfrak{h} := \chi_{B_{\varrho}(y)}$. Thus, for every $x \in \mathbb{R}^d$, the following estimate holds true
\begin{align*}
\chi_{\Omega_{\varrho}(y)}(x) \le \mathfrak{h}(x) \le \mathcal{M} \mathfrak{h}(x) \le \left(\mathcal{M} \mathfrak{h}\right)^{\vartheta}(x) \le 1,
\end{align*}
for any $\vartheta \in (0,1)$. Therefore, we readily obtain that
\begin{align}\label{est-1}
T(y,\varrho) & = \frac{1}{\psi(y,\varrho)} \int_{\mathbb{R}^d}[\chi_{\Omega}\mathcal{M}_{\alpha}\mathbb{U}(x)]^s \mathfrak{h}(x) dx \notag \\
& \le \frac{1}{\psi(y,\varrho)} \int_{\mathbb{R}^d}[\chi_{\Omega}\mathcal{M}_{\alpha}\mathbb{U}(x)]^s [\mathcal{M}\mathfrak{h}(x)]^{\vartheta} dx.
\end{align}
Thanks to \cite[Proposition 2]{CR80}, for $\vartheta \in (0,1)$ one has $\left(\mathcal{M}\mathfrak{h}\right)^{\vartheta} \in \mathcal{A}_1 \subset \mathcal{A}_{\infty}$. For this reason, one may apply~\eqref{est-0} with $\omega = \left(\mathcal{M}\mathfrak{h}\right)^{\vartheta}$ to arrive at
\begin{align}\notag 
\int_{\mathbb{R}^d} \left[\chi_{\Omega}\mathcal{M}_{\alpha}\mathbb{U}(x)\right]^{\mathfrak{s}} \left(\mathcal{M}\mathfrak{h}(x)\right)^{\vartheta} dx \le C\int_{\mathbb{R}^d} \left[\chi_{\Omega}\mathcal{M}_{\alpha}\mathbb{F}_{\sigma}(x)\right]^{\mathfrak{s}} \left(\mathcal{M}\mathfrak{h}(x)\right)^{\vartheta} dx.
\end{align}
Plugging this estimate into~\eqref{est-1}, it gives
\begin{align}\label{est-2}
T(y,\varrho) & \le \frac{C}{\psi(y,\varrho)} \int_{\mathbb{R}^d} \left[\chi_{\Omega}\mathcal{M}_{\alpha}\mathbb{F}_{\sigma}(x)\right]^{\mathfrak{s}} \left(\mathcal{M}\mathfrak{h}(x)\right)^{\vartheta} dx.
\end{align}
Let us now denote $\varrho_n = 2^n\varrho$ for $n \in \mathbb{N}$ and use a dyadic decomposition of $\mathbb{R}^d$ by 
\begin{align*}
\mathbb{R}^d = B_{\varrho_1}(y) \cup \left(\bigcup_{n=1}^{\infty} B_{\varrho_{n+1}}(y) \setminus B_{\varrho_{n}}(y) \right).
\end{align*}
Thus, we are able to estimate~\eqref{est-2} as below
\begin{align}\label{est-3}
T(y,\varrho) & \le \frac{C}{\psi(y,\varrho)}\int_{B_{\varrho_1}(y)} \left[\chi_{\Omega}\mathcal{M}_{\alpha}\mathbb{F}_{\sigma}(x)\right]^{\mathfrak{s}} \left(\mathcal{M}\mathfrak{h}(x)\right)^{\vartheta} dx \notag \\
& \qquad \qquad + \frac{C}{\psi(y,\varrho)} \sum_{n=1}^{\infty}\int_{B_{\varrho_{n+1}}(y) \setminus B_{\varrho_n}(y)} \left[\chi_{\Omega}\mathcal{M}_{\alpha}\mathbb{F}_{\sigma}(x)\right]^{\mathfrak{s}} \left(\mathcal{M}\mathfrak{h}(x)\right)^{\vartheta} dx.
\end{align}
Furthermore, for every $x \in B_{\varrho_{n+1}}(y) \setminus B_{\varrho_n}(y)$, one has $\varrho_n \le |x - y| < \varrho_{n+1}$, and for all $z \in B_{\varrho}(y)$, it is easy to check that
\begin{align*}
|z - x| \ge |x-y| - |z - y| > \varrho_n - \varrho,
\end{align*}
which ensures that $B_r(x) \cap B_{\varrho}(y)$ is empty for every $r \le \varrho_n-\varrho$. It allows to carry over the quantity
\begin{align}\label{est-Mh}
\mathcal{M}\mathfrak{h}(x) & = \sup_{r>0} \fint_{B_r(x)} \chi_{B_{\varrho}(y)}(z)dz = \sup_{r>0} \frac{|B_r(x) \cap B_{\varrho}(y)|}{|B_r(x)|} \notag\\
& =  \sup_{r>\varrho_n-\varrho} \frac{|B_r(x) \cap B_{\varrho}(y)|}{|B_r(x)|} \notag \\
& = \max\left\{\sup_{\varrho_n-\varrho < r < \varrho_{n+1}+\varrho} \frac{|B_r(x) \cap B_{\varrho}(y)|}{|B_r(x)|}; \sup_{r \ge \varrho_{n+1}+\varrho} \frac{|B_r(x) \cap B_{\varrho}(y)|}{|B_r(x)|}\right\}.
\end{align}
On the other hand, for all $z \in B_{\varrho}(y)$ and $r \ge \varrho + \varrho_{n+1}$, a simple computation we have
\begin{align*}
|z - x| \le |z-y| + |y - x| < \varrho + \varrho_{n+1} \le r,
\end{align*}
which concludes
\begin{align*}
\sup_{r \ge \varrho_{n+1}+\varrho} \frac{|B_r(x) \cap B_{\varrho}(y)|}{|B_r(x)|} = \sup_{r \ge \varrho_{n+1}+\varrho} \frac{|B_{\varrho}(y)|}{|B_r(x)|} = \left(\frac{\varrho}{\varrho_{n+1}+\varrho}\right)^d = \frac{1}{(2^{n+1}+1)^d} \le \frac{1}{2^{nd}}. 
\end{align*}
Moreover, as we have
\begin{align}\notag
\sup_{\varrho_n-\varrho < r < \varrho_{n+1}+\varrho} \frac{|B_r(x) \cap B_{\varrho}(y)|}{|B_r(x)|} \le \sup_{\varrho_n-\varrho < r < \varrho_{n+1}+\varrho} \frac{|B_{\varrho}(y)|}{|B_r(x)|} = \frac{1}{(2^{n}-1)^d} \le \frac{2^d}{2^{nd}},
\end{align}
so merging the last two inequalities into~\eqref{est-Mh}, it yields
\begin{align*}
\mathcal{M}\mathfrak{h}(x) \le \frac{2^d}{2^{nd}}, \mbox{ for all } x \in B_{\varrho_{n+1}}(y) \setminus B_{\varrho_n}(y).
\end{align*}
Reabsorbing this estimate into the right-hand side of~\eqref{est-3} and using assumption~\eqref{cond-psi-2} on $\psi$ in hand, one gets that
\begin{align}\label{est-4}
T(y,\varrho) & \le \frac{C}{\psi(y,\varrho)}\int_{B_{\varrho_1}(y)} \left[\chi_{\Omega}\mathcal{M}_{\alpha}\mathbb{F}_{\sigma}(x)\right]^{\mathfrak{s}} dx \notag \\ 
& \qquad + C \sum_{n=1}^{\infty} \frac{1}{\psi(y,\varrho)} \frac{1}{2^{nd\vartheta}} \int_{B_{\varrho_{n+1}}(y)} \left[\chi_{\Omega}\mathcal{M}_{\alpha}\mathbb{F}_{\sigma}(x)\right]^{\mathfrak{s}} dx \notag \\
& \le \frac{C\beta_0}{\psi(y,2\varrho)}\int_{B_{\varrho_1}(y)} \left[\chi_{\Omega}\mathcal{M}_{\alpha}\mathbb{F}_{\sigma}(x)\right]^{\mathfrak{s}} dx \notag \\
& \qquad +  \sum_{n=1}^{\infty} \frac{C \beta_0^{n+1}}{2^{nd\vartheta}} \frac{1}{\psi(y,\varrho_{n+1})} \int_{B_{\varrho_{n+1}}(y)} \left[\chi_{\Omega}\mathcal{M}_{\alpha}\mathbb{F}_{\sigma}(x)\right]^{\mathfrak{s}} dx \notag \\
& \le C \beta_0 \left[1 + \sum_{n=1}^{\infty} \left(\frac{\beta_0}{2^{d\vartheta}}\right)^n\right] \|\mathcal{M}_{\alpha}\mathbb{F}_{\sigma}\|_{\mathrm{M}^{\mathfrak{s},\psi}(\Omega)}^{\mathfrak{s}}.
\end{align}
Since $\beta_0 \in (1,2^d)$, it is now possible to take $\vartheta = \vartheta(d,\beta_0) \in (0,1)$ such that 
$$\frac{\beta_0}{2^{d \vartheta}}<1 \Leftrightarrow \log_2(\beta_0)/d<\vartheta,$$ 
and it ensures that the series in~\eqref{est-4} is finite. The proofs of~\eqref{GF-norm} is complete.
\end{proof}

\section{Global regularity for asymptotically regular problems}

This section is devoted to the proof of regularity estimate for weak solution to the asymptotically regular problem~\eqref{eq:asym}. As mentioned above, it is a consequence of global gradient bounds proved in preceding section. This proof exploits the level-set argument similar to that used for original Stokes problem. The assertion of Theorem~\ref{theo:asym} continues to hold in generalized function spaces.

\begin{proof}[Proof of Theorem~\ref{theo:asym}]
First, let us introduce a new matrix-valued function $\Theta_{\sigma}$ which defined in $\Omega \times \mathbb{R}^{d^2}$ as below
\begin{align}\notag 
& \Theta_{\sigma}(x,\nu) := \frac{\mathbf{a}(x,\nu)-\mathbf{A}(x,\nu)}{(\sigma^2+|\nu|^2)^{\frac{p-1}{2}}}, \quad (x,\nu) \in \Omega \times \mathbb{R}^{d^2}.
\end{align}
The assumption~\eqref{cond:asym} ensures that there exists $M_0>0$ such that
\begin{align}\label{ineq-delta}
 |\Theta_{\sigma}(x,\nu)| \le \delta_0,
\end{align}
for every $\nu \in \mathbb{R}^{d^2} \setminus B_{M_0}$, uniformly with respect to $x \in \Omega$. Also, for each $x \in \Omega$, let us consider the Poisson integral
\begin{align}\notag 
\mathbb{P}[\Theta_{\sigma}(x,\cdot)](y) & := \int_{\partial B_{M_0}} \Theta_{\sigma}(x,\nu) \mathbb{K}(y,\nu) dS(\nu),
\end{align}
for all $y \in B_{M_0}$. Here, $\mathbb{K}$ stands for the Poisson kernel given by
\begin{align}\notag 
\mathbb{K}(y,\nu) := \left(M_0 \omega_{n-1}\right)^{-1}\frac{M_0^2 - |y|^2}{|y-\nu|^n},
\end{align}
where $\omega_{n-1}$ denotes the surface area of $\partial B_1$ in $\mathbb{R}^{d^2}$. Now, we consider
\begin{align}\notag 
\Pi_{\sigma}(x,y) := \begin{cases} \Theta_{\sigma}(x,y), & \quad \mbox{ if } |y| \ge M_0,\\ \mathbb{P}[\Theta_{\sigma}(x,\cdot)](y), & \quad \mbox{ if } |y| < M_0. \end{cases}
\end{align}
Making use of~\eqref{ineq-delta} and the maximum principle, it allows us to verify that $\Pi_{\sigma}$ is Carath\'eodory map and moreover, we pattern the proof on arguments from~\cite{BOW2015} to obtain 
\begin{align}\label{est-Pi}
|\Pi_{\sigma}(x,y)| \le \delta_0, \quad \forall y \in \mathbb{R}^{d^2},
\end{align}
uniformly with respect to $x \in \Omega$. Keep in mind the notation already introduced, it allows us to establish the relation between $\mathbf{a}$ and $\mathbf{A}$ as following
\begin{align}\notag 
\mathbf{a}(x,y) & = \mathbf{A}(x,y) + (\sigma^2+|y|^2)^{\frac{p-1}{2}} \Pi_{\sigma}(x,y) + (\sigma^2+|y|^2)^{\frac{p-1}{2}}\chi_{\{|y|<M_0\}} \left(\Theta_{\sigma}(x,y)-\Pi_{\sigma}(x,y)\right).
\end{align}
At this point, for the ease of the reader we further consider vector-valued function $\mathcal{E}$ as
\begin{align}\label{def:E}
\mathcal{E}(x,y) := \mathbf{A}(x,y) + (\sigma^2+|y|^2)^{\frac{p-1}{2}} \Pi_{\sigma}(x,\nabla u).
\end{align}
It is now easy to see that $\mathcal{E}(x,y)$ is differentiable with respect to $y$, and
\begin{align}\notag 
\partial_y \mathcal{E}(x,y) & = \partial_y \mathbf{A}(x,y) + (p-1)(\sigma^2+|y|^2)^{\frac{p-3}{2}} \langle \Pi_{\sigma}(x,\nabla u),  y\rangle, 
\end{align}
which from~\eqref{est-Pi} and~\eqref{cond:A1}, we deduce that
\begin{align}\label{est-E-1}
|\mathcal{E}(x,y)| + |\partial_y \mathcal{E}(x,y)| |y| & \le |\mathbf{A}(x,y)| + |\partial_y \mathbf{A}(x,y)| |y| + p\delta_0 (\sigma^2+|y|^2)^{\frac{p-1}{2}} \notag \\
& \le (\Upsilon + p) (\sigma^2+|y|^2)^{\frac{p-1}{2}},
\end{align}
for every $\delta_0 \in (0,1)$. In addition, for every $y_1, y_2 \in \mathbb{R}^{d^2}$, one gets
\begin{align}\label{est-E0}
\langle \mathcal{E}(x,y_1) - \mathcal{E}(x,y_2), y_1 - y_2\rangle & = \langle \mathbf{A}(x,y_1) - \mathbf{A}(x,y_2), y_1 - y_2\rangle \notag \\
&  +  \left[(\sigma^2+|y_1|^2)^{\frac{p-1}{2}}  - (\sigma^2+|y_2|^2)^{\frac{p-1}{2}}\right] \langle \Pi_{\sigma}(x,\nabla u), y_1 - y_2\rangle .
\end{align}
A closer look at \cite[Lemma 2.1]{H92} finds a constant $c>0$ such that
\begin{align*}
\left|(\sigma^2+|y_1|^2)^{\frac{p-1}{2}}  - (\sigma^2+|y_2|^2)^{\frac{p-1}{2}}\right| & \le c \left(\sigma^2 + |y_1|^2 + |y_2|^2 \right)^{\frac{p-2}{2}}|y_1 - y_2|.
\end{align*}
Using~\eqref{cond:A2},~\eqref{est-Pi} and~\eqref{est-E0} in the following way:
\begin{align}\label{est-E-2}
\langle \mathcal{E}(x,y_1) - \mathcal{E}(x,y_2), y_1 - y_2\rangle & \ge \left(\Upsilon^{-1}-2\delta_0 c\right) \left(\sigma^2 + |y_1|^2 + |y_2|^2 \right)^{\frac{p-2}{2}}|y_1 - y_2|^2.
\end{align}
We now take~\eqref{def:E} into account to clarify $\mathbf{a}(x,\nabla \mathbf{u}_{\mathbf{a}})$ as below
\begin{align}\notag 
\mathbf{a}(x,\nabla \mathbf{u}_{\mathbf{a}}) & = \mathcal{E}(x,\nabla \mathbf{u}_{\mathbf{a}}) + (\sigma^2+|\nabla \mathbf{u}_{\mathbf{a}}|^2)^{\frac{p-1}{2}}\chi_{\{|\nabla \mathbf{u}_{\mathbf{a}}|<M_0\}} \left(\Theta_{\sigma}(x,\nabla \mathbf{u}_{\mathbf{a}})-\Pi_{\sigma}(x,\nabla \mathbf{u}_{\mathbf{a}})\right).
\end{align}
Since $(\mathbf{u}_{\mathbf{a}},\mathrm{\pi}_{\mathbf{a}}) \in W^{1,p}_{\mathbf{g},\mathrm{div}}(\Omega;\mathbb{R}^d) \times L^{p'}_{\mathrm{int}}(\Omega)$ is a pair of solutions to system of type~\eqref{eq:asym}, to be precise one has
\begin{align}\label{eq:asym-2}
\begin{cases}
-\mathrm{div} \left(\mathcal{E}(x,\nabla \mathbf{u}_{\mathbf{a}})\right) + \nabla \mathrm{\pi}_{\mathbf{a}} &= -\mathrm{div}\left(\mathbf{b}(x,\mathbf{f})\right) \quad \  \text{in}\ \Omega, \\
 \hspace{1.5cm} \mathrm{div} (\mathbf{u}_{\mathbf{a}}) &= \ 0  \hspace{2.45cm} \ \text{in}\ \Omega, \\
 \hspace{1.9cm} \mathbf{u}_{\mathbf{a}} & = \ \mathbf{g} \hspace{2.55cm} \text{on}\ \partial\Omega,
 \end{cases}
\end{align}
where 
\begin{align}\label{est-b}
\mathbf{b}(x,\mathbf{f}) = \mathbf{B}(x,\mathbf{f}) + (\sigma^2+|\nabla \mathbf{u}_{\mathbf{a}}|^2)^{\frac{p-1}{2}}\chi_{\{|\nabla \mathbf{u}_{\mathbf{a}}|<M_0\}} \left(\Pi_{\sigma}(x,\nabla \mathbf{u}_{\mathbf{a}})-\Theta_{\sigma}(x,\nabla \mathbf{u}_{\mathbf{a}})\right).
\end{align}
Thanks to~\eqref{est-Pi}, we write
\begin{align}\label{est-b1}
\left|(\sigma^2+|\nabla \mathbf{u}_{\mathbf{a}}|^2)^{\frac{p-1}{2}}\chi_{\{|\nabla \mathbf{u}_{\mathbf{a}}|<M_0\}}\Pi_{\sigma}(x,\nabla \mathbf{u}_{\mathbf{a}})\right| \le 2\delta_0 \left(\sigma^2+M_0^2\right)^{\frac{p-1}{2}},
\end{align}
and on the other hand, assumption~\eqref{cond:asym} gives us that
\begin{align}\notag
\left|\mathbf{a}(x,\nu)-\mathbf{A}(x,\nu)\right| \le \epsilon(|\nu|)(1 + |\nu|^{p-1}), \mbox{ and } \limsup_{\lambda \to \infty} \epsilon(\lambda) \le \delta_0,
\end{align}
where $\epsilon: \mathbb{R}^+ \to \mathbb{R}^+$ is defined by
\begin{align*}
\epsilon(|\nu|) = \sup_{x\in \Omega} \frac{\left|\mathbf{a}(x,\nu)-\mathbf{A}(x,\nu)\right|}{(\sigma^2+|\nu|^2)^{\frac{p-1}{2}}}.
\end{align*}
For this reason, one may estimate as below
\begin{align}\label{est-b2}
\left|(\sigma^2+|\nabla \mathbf{u}_{\mathbf{a}}|^2)^{\frac{p-1}{2}}\chi_{\{|\nabla \mathbf{u}_{\mathbf{a}}|<M_0\}}\Theta_{\sigma}(x,\nabla \mathbf{u}_{\mathbf{a}})\right| & = \left|\chi_{\{|\nabla \mathbf{u}_{\mathbf{a}}|<M_0\}}\left[\mathbf{a}(x,\nabla \mathbf{u}_{\mathbf{a}})-\mathbf{A}(x,\nabla \mathbf{u}_{\mathbf{a}})\right]\right| \notag \\
& \le \chi_{\{|\nabla \mathbf{u}_{\mathbf{a}}|<M_0\}}\epsilon(|\nabla \mathbf{u}_{\mathbf{a}}|)(1 + |\nabla \mathbf{u}_{\mathbf{a}}|^{p-1}) \notag \\
& \le \|\epsilon\|_{\infty} \left(1 + M_0^{p-1}\right).
\end{align}
Plugging~\eqref{est-b1} and~\eqref{est-b2} into~\eqref{est-b}, it yields
\begin{align}\label{est-b3}
\left|\mathbf{b}(x,\mathbf{f})\right| \le C(p,M_0,\|\epsilon\|_{\infty},\sigma) \left(1 + |\mathbf{f}|^2\right)^{\frac{p-1}{2}}.
\end{align}
By the preceding estimates~\eqref{est-E-1},~\eqref{est-E-2} and~\eqref{est-b3}, for $0<\delta_0<\frac{1}{4c\Upsilon}$, we are able to find $\tilde{\Upsilon}>0$ such that
\begin{align*}
\begin{cases} |\mathcal{E}(x,y)| + |\partial_y \mathcal{E}(x,y)| |y| \le \tilde{\Upsilon} (1 + |y|^2)^{\frac{p-1}{2}}, \\
\langle \mathcal{E}(x,y_1) - \mathcal{E}(x,y_2), y_1 - y_2\rangle \ge {\tilde{\Upsilon}}^{-1} \left(1 + |y_1|^2 + |y_2|^2 \right)^{\frac{p-2}{2}}|y_1 - y_2|^2,\\
\left|\mathbf{b}(x,\mathbf{f})\right| \le \tilde{\Upsilon} \left(1 + |\mathbf{f}|^2\right)^{\frac{p-1}{2}}.\end{cases}
\end{align*} 
Finally, applying Theorem~\ref{theo:main} and Theorem~\ref{theo:improv} to solution $(\mathbf{u}_{\mathbf{a}},\mathrm{\pi}_{\mathbf{a}})$ of problem~\eqref{eq:asym-2}, it leads to the desired estimate ~\eqref{main-asym}. We have thus completed the proof of Theorem~\ref{theo:asym}.
\end{proof}

\section*{Declarations}
Ethics approval and consent to participate: yes.\\

\noindent
Consent for publication: yes.\\

\noindent
Availability of data and materials: Data sharing not applicable to this article as no datasets were generated or analysed during the current study.\\


%
%
%

\section*{Conflict of Interest}
The authors declared that they have no conflict of interest.

\section*{Acknowledgement}
This research is funded by Vietnam National Foundation for Science and Technology Development (NAFOSTED) under grant number 101.02-2021.17.

\end{document}